\documentclass{amsart}
\usepackage{amssymb}
\usepackage{pb-diagram}
\usepackage[]{hyperref}
\usepackage{verbatim}
\usepackage{xcolor}
\usepackage{multirow}
\usepackage{tabmac}

\usepackage[colorinlistoftodos]{todonotes}
\usepackage{tikz}
\usetikzlibrary{cd}

\newtheorem{thm}{Theorem}
\newtheorem{conj}[thm]{Conjecture}
\newtheorem{lem}[thm]{Lemma}
\newtheorem{prop}[thm]{Proposition}
\newtheorem{cor}[thm]{Corollary}
\theoremstyle{remark}
\newtheorem{rem}[thm]{Remark}
\newtheorem{ex}[thm]{Example}
\numberwithin{thm}{section}

\newcommand{\ba}{\mathbf{a}}

\newcommand{\bF}{\mathbf{F}}
\newcommand{\bi}{\mathbf{i}}
\newcommand{\bV}{\mathbf{V}}
\newcommand{\bZ}{\mathbb{Z}}
\newcommand{\cB}{\mathcal{B}}

\newcommand{\hmu}{{\hat{\mu}}}

\newcommand{\la}{\lambda}
\newcommand{\lad}{{\lambda^\bullet}}
\newcommand{\La}{\Lambda}

\newcommand{\mud}{{\mu^{\bullet}}}
\newcommand{\nud}{{\nu^{\bullet}}}
\newcommand{\rhod}{{\rho^\bullet}}
\newcommand{\Spr}{\mathrm{Spr}}

\newcommand{\Xd}{{X^\bullet}}
\newcommand{\wbul}{{w^\bullet}}

\newcommand{\bY}{\mathbb{Y}}

\newcommand{\bC}{\mathbb{C}}
\newcommand{\bQ}{\mathbb{Q}}
\newcommand{\cK}{\mathcal{K}}
\newcommand{\cL}{\mathcal{L}}
\newcommand{\cW}{\mathcal{W}}
\newcommand{\cG}{\mathcal{G}}
\newcommand{\cF}{\mathcal{F}}
\newcommand{\cO}{\mathcal{O}}

\newcommand{\gl}{\mathfrak{gl}}

\newcommand{\bW}{\mathbf{W}}

\newcommand{\wt}{\mathrm{wt}}
\newcommand{\End}{\mathrm{End}}
\newcommand{\Fl}{\mathrm{Fl}}

\newcommand{\Hom}{\mathrm{Hom}}

\newcommand{\Sym}{\mathrm{Sym}}

\newcommand{\gd}{\gamma^\bullet}

\newcommand{\pair}[2]{\langle #1\,,\,#2\rangle}

\newcommand{\pairaux}[2]{\left(#1\,,\,#2\right)}

\newcommand{\fixit}[1]{\texttt{\color{red}{*** #1 ***}}}
\newcommand{\dom}{\trianglerighteq}
\newcommand{\pr}{\mathrm{pr}}

\newcommand{\Z}{\bZ}
\newcommand{\C}{\bC}
\newcommand{\al}{\alpha}
\newcommand{\alb}{\overline{\alpha}}
\newcommand{\Pb}{\overline{P}}
\newcommand{\Lambdab}{\overline{\Lambda}}
\newcommand{\fsl}{\mathfrak{sl}}
\newcommand{\del}{\partial}
\newcommand{\K}{\mathbb{K}}
\newcommand{\Lb}{\overline{L}}

\newcommand{\Utor}{\mathbf{\ddot U}_{q,t}(\fsl_{r})}

\newcommand{\gch}{\mathrm{ch}}

\newcommand{\cS}{\mathcal{S}}
\newcommand{\hcS}{\widehat{\mathcal{S}}}
\newcommand{\hR}{\widehat{R}}

\newcommand{\hbi}{\hat{\bi}}
\newcommand{\hba}{\hat{\ba}}
\newcommand{\Out}{\mathrm{Out}}

\newcommand{\cat}{\mathrm{cat}}
\newcommand{\red}{\mathrm{red}}
\newcommand{\id}{\mathrm{id}}

\renewcommand{\comment}[1]{}

\title[Quiver Hall-Littlewood functions]{Quiver Hall-Littlewood functions and Kostka-Shoji polynomials}

\date{\today}
\author{Daniel Orr}
\address{
Department of Mathematics (MC 0123),
460 McBryde Hall, Virginia Tech,
225 Stanger St.,
Blacksburg, VA 24061 USA}
\email{dorr@vt.edu}
\author{Mark Shimozono}
\address{
Department of Mathematics (MC 0123),
460 McBryde Hall, Virginia Tech,
225 Stanger St.,
Blacksburg, VA 24061 USA}
\email{mshimo@math.vt.edu}

\begin{document}

\maketitle

\begin{abstract}
For any triple $(i,a,\mu)$ consisting of a vertex $i$ in a quiver $Q$, a positive integer $a$, and a dominant $GL_a$-weight $\mu$, we define a quiver current $H^{(i,a)}_\mu$ acting on the tensor power $\Lambda^Q$ of symmetric functions over the vertices of $Q$. These provide a quiver generalization of parabolic Garsia-Jing creation operators in the theory of Hall-Littlewood symmetric functions. For a triple $(\bi,\ba,\mu(\bullet))$ of sequences of such data, we define the quiver Hall-Littlewood function $H^{\bi,\ba}_{\mu(\bullet)}$ as the result of acting on $1\in\Lambda^Q$ by the corresponding sequence of quiver currents. The quiver Kostka-Shoji polynomials are the expansion coefficients of $H^{\bi,\ba}_{\mu(\bullet)}$ in the tensor Schur basis. These polynomials include the Kostka-Foulkes polynomials and parabolic Kostka polynomials (Jordan quiver) and the Kostka-Shoji polynomials (cyclic quiver) as special cases.

We show that the quiver Kostka-Shoji polynomials are graded multiplicities in the equivariant Euler characteristic of a vector bundle (determined by $\mu(\bullet)$) on Lusztig's convolution diagram determined by the sequences $\bi,\ba$. For certain compositions of currents we conjecture higher cohomology vanishing of the associated vector bundle on Lusztig's convolution diagram. For quivers with no branching we propose an explicit positive formula for the quiver Kostka-Shoji polynomials in terms of catabolizable multitableaux.

We also relate our constructions to $K$-theoretic Hall algebras, by realizing the quiver Kostka-Shoji polynomials as natural structure constants and showing that the quiver currents provide a symmetric function lifting of the corresponding shuffle product. In the case of a cyclic quiver, we explain how the quiver currents arise in Saito's vertex representation of the quantum toroidal algebra of type $\fsl_r$.
\end{abstract}


\section{Introduction}
Consider the space $\La^Q = \otimes_{i\in Q_0} \La^{(i)}$ of \textit{quiver symmetric functions}, the tensor power of symmetric functions $\La$ with one factor for each vertex $i\in Q_0$ of a quiver $Q$, over a coefficient ring with a parameter $t_b$ for each arrow $b\in Q_1$ in the quiver. 
Given a triple $(i,a,\mu)$ where $i\in Q_0$ is a vertex, $a\in \Z_{>0}$, and $\mu\in X_+(GL_a)$ is a dominant integral $GL_a$-weight, we introduce the \textit{quiver current} $H^{(i,a)}_\mu\in \mathrm{End}(\La^Q)$. Quiver currents are generalizations of parabolic Garsia-Jing creation operators \cite{G} \cite{J} \cite{SZ}. Compositions of such operators are indexed by triples $(\bi,\ba,\mu(\bullet))$ where $\bi$ is a list of vertices, $\ba$ is a list of positive integers, and $\mu(\bullet)$ is a list of weights. When acting on the vacuum vector, such a composite operator creates the \textit{Hall-Littlewood quiver symmetric function}
$H^{\bi,\ba}_{\mu(\bullet)}\in\La^Q$. Their expansion coefficients $\cK^{\bi,\ba}_{\lad,\mu(\bullet)}(t_{Q_1})$ in the tensor Schur basis $s_\lad$ are polynomials in the arrow variables with integer coefficients. We call these the \textit{quiver Kostka-Shoji polynomials}. When all arrow variables are set to a single parameter, these are instances of Panyushev's generalized Kostka polynomials \cite{P}.

\subsection{Special cases}

\subsubsection{Single loop} Consider the quiver with one vertex and a loop.
If all weights in $\mu(\bullet)$ are single rows, $\cK^{\bi,\ba}_{\lad,\mu(\bullet)}(t_{Q_1})$ is the classical Kostka-Foulkes polynomial \cite{B} \cite{LS} \cite{Mac}. These polynomials give local intersection cohomology for the nullcone \cite{L:nullcone}.

If all weights are rectangles of a fixed width, $\cK^{\bi,\ba}_{\lad,\mu(\bullet)}(t_{Q_1})$ is a graded multiplicity at an irreducible $GL_n$-character in the coordinate ring of a nilpotent adjoint orbit closure in $\mathfrak{gl}_n$ \cite{S:poset} \cite{W}. 

If all weights are rectangles $\cK^{\bi,\ba}_{\lad,\mu(\bullet)}(t_{Q_1})$ is a graded multiplicity in tensor products of affine type A Kirillov-Reshetikhin modules \cite{KSS} \cite{ScW} \cite{S:poset} \cite{S:multi} \cite{S:affine} \cite{W}.

In general $\cK^{\bi,\ba}_{\lad,\mu(\bullet)}(t_{Q_1})$
is a parabolic Kostka polynomial \cite{Bro} \cite{SW}.

\subsubsection{$r$-vertex cyclic quiver}
This includes the case of the single loop quiver ($r=1$).
For a very specific sequence of currents, the polynomials $\cK^{\bi,\ba}_{\lad,\mu(\bullet)}(t_{Q_1})$ were studied by Finkelberg and Ionov \cite{FI} and earlier by Shoji in connection with Green's polynomials for complex reflection groups \cite{Sh,Sh2} (see Examples \ref{X:cycle} and \ref{X:cyclic vanishing}).

For $r=2$ these polynomials have an interpretation in intersection cohomology for the enhanced nullcone \cite{AH} and the mirabolic affine Grassmannian~\cite{FGT}.

\subsubsection{$A_2$ quiver}
For the quiver with two vertices $Q_0=\{0,1\}$ and a single edge $(0,1)$, for special sequences of currents a formula for $\cK^{\bi,\ba}_{\lad,\mu(\bullet)}(t_{Q_1})$ was proved in \cite{Cr}. The answer is a single power of $t_{01}$ times a truncated Littlewood-Richardson coefficient.

\subsection{Lusztig's convolution diagram}
Associated with a composition of currents is a vector bundle on Lusztig's convolution diagram \cite{L:conv}, which is itself a vector bundle over a product of partial flag varieties, one for each quiver vertex.\footnote{Lusztig's construction is defined for loopless quivers. We allow loops but add a condition which enforces nilpotence.} The quiver Kostka-Shoji polynomials are the arrow-graded isotypic components of the (virtual) quiver $GL$-module afforded by the Euler characteristic of this vector bundle. For certain compositions of currents we conjecture higher cohomology vanishing of the vector bundle on Lusztig's convolution diagram (Conjecture~\ref{conj:positivity}), which implies the positivity of quiver Kostka-Shoji polynomials. Higher vanishing is known in some cases by a result of Panyushev~\cite{P}, which we explain in Remark~\ref{R:vanish}.

\subsection{Combinatorics}
For quivers with no branching (that is, whose connected components are directed cycles or directed paths) we propose an explicit positive formula for the quiver Kostka-Shoji polynomials (Conjecture \ref{conj:cat}) in terms of \textit{catabolizable multitableaux}.
For the single loop quiver this reduces to \cite[Conjecture 27]{SW}.

\subsection{$K$-theoretic Hall algebras}
The direct sum of equivariant $K$-groups of quiver representation spaces over all possible dimension vectors carries a natural associative algebra structure, which we define in Section~\ref{S:shuffle} by adapting Kontsevich and Soibelman's definition of cohomological Hall algebras \cite{KS}. This turns out to be identical to Lusztig's convolution product \cite{L:conv}, just applied to equivariant $K$-groups rather than perverse sheaves (see Remark~\ref{R:lusztig-conv}). The multiplication in this ring, the (non-preprojective) $K$-theoretic Hall algebra, is given by a shuffle product directly analogous to that of \cite{KS}, however containing extra equivariance from the arrow parameters. 

We express the quiver Hall-Littlewood series as iterated shuffle products (Proposition~\ref{P:quiver-kostka-shuffle}), which leads to another interpretation of quiver Kostka-Shoji polynomials: they are structure constants with respect to tensor Schur polynomials in the $K$-theoretic Hall algebra. We also show that the quiver currents provide a natural symmetric function lifting of the multiplication in the $K$-theoretic Hall algebra (Proposition~\ref{P:quiver-current-shuffle}) and its preprojective $(q,t)$-version (Proposition~\ref{P:qt-quiver-current-shuffle}). Preprojective cohomological Hall algebras were introduced by Yang and Zhao \cite{YZ} for any algebraic oriented cohomology theory; earlier, the $K$-theoretic case of these algebras were studied extensively for single loop \cite{SV} \cite{FT} \cite{Ne-rev} and cyclic quivers~\cite{Ne}. We define the corresponding $(q,t)$-versions of our quiver currents in Section~\ref{SS:qt-preprojective}.

\subsection{Quantum toroidal $\fsl_r$}
For cyclic quivers the $(q,t)$-quiver currents with $a=1$ give the vertex representation \cite{Sa} of the (positive part of the) quantum toroidal algebra $\Utor$ \cite{GKV}. This connection is explained in Appendix~\ref{S:toroidal}. This provides a realization of Kostka-Shoji polynomials \cite{Sh2} in the representation theory of quantum toroidal algebras which we intend to pursue further in future work.

The vertex representation of $\Utor$ arises geometrically via the equivariant $K$-groups of Nakajima varieties for the cyclic quiver with one-dimensional framing space \cite{VV} \cite{Na} (see also \cite{FJMM} and \cite{T} for the connections between various $\Utor$-representations). It is interesting to speculate about the relationship between our symmetric function operators and Nakajima's geometric construction of quantum loop algebra representations for general quivers \cite{N}. However, the connection to quiver Hall-Littlewood functions in the case of cyclic quivers requires separate parameters $(q,t)$, followed by the specialization $q=0$ (see the right-hand side of \eqref{E:qt-current}). For general quivers we do not know the meaning of this specialization in the setting of Nakajima varieties.

\comment{
In \cite{Ne}, the preprojective $K$-theoretic Hall algebra for the cyclic quiver, namely the corresponding $(q,t)$-shuffle algebra, is shown to be isomorphic to $\Utor$. The $\Utor$-representations constructed by Nakajima's method \cite{N} are studied extensively in \cite{Ne}. This includes the vertex representation \cite{Sa}, as mentioned above, though our realization in terms of quiver symmetric functions $\Lambda^{Q}$ is different.

More generally, it is a special case of \cite[Theorem~5.4]{YZ} that the preprojective $K$-theoretic Hall algebra of any quiver acts on the equivariant $K$-groups of Nakajima quiver varieties, though an explicit presentation of the former algebra is not available in general. 

We intend to clarify this point in future work.
}

The $(q,t)$-version of the cyclic quiver current (for $a=1$, see \eqref{E:qt-current}) is the natural analog of certain operators generating part of the action of the elliptic Hall algebra (also known as ``quantum toroidal $\mathfrak{gl}_1$'') on symmetric functions~\cite{SV}, which are known to play a fundamental and important role in Macdonald theory \cite{GHT} \cite{BGLX} \cite{GN}. We hope that the $(q,t)$-currents will play an equally important role in a theory of cyclic quiver Macdonald symmetric functions, possibly those of \cite{Sh-mac}, forming a basis of cyclic quiver symmetric functions $\Lambda^{Q}$ involving both parameters $(q,t)$ and which specialize to the quiver Hall-Littlewood functions at $q=0$.\footnote{Another candidate for cyclic quiver Macdonald symmetric functions would seem to be given by the wreath Macdonald polynomials of \cite{Hai} \cite{BF}; however, remarks in \cite[\S7.2.4]{Hai} indicate that these polynomials at $q=0$ do not coincide with Shoji's \cite{Sh}.}

\subsubsection*{Acknowledgements}
We thank Michael Finkelberg, Mee Seong Im, and Eric Vasserot for stimulating discussions and helpful correspondence.
The authors were partially supported by NSF grant DMS-1600653.

\section{Quiver Hall-Littlewood series}

\subsection{Quiver}\label{SS:data}

Let $Q=(Q_0,Q_1)$ be a quiver (finite directed graph) with vertex set $Q_0$ and arrow set $Q_1$. We write $ha\in Q_0$ (resp. $ta\in Q_0$) for the head (resp. tail) of the arrow $a\in Q_1$: pictorially $ta \overset{a}{\rightarrow} ha$. Loops ($a\in Q_1$ such that $ha=ta$) and multiple edges ($a\ne b\in Q_1$ with $ha=hb$ and $ta=tb$) are allowed.

Let $\Z^{Q_0} = \bigoplus_{i\in Q_0} \bZ f^{(i)}$ be the lattice of virtual dimension vectors. We write its elements as $\nud = \sum_{i\in Q_0} \nu^{(i)} f^{(i)}$ where $f^{(i)}\in \bZ^{Q_0}$ is the $i$-th standard basis vector and $\nu^{(i)}$ is the coefficient of $\nu$ at $f^{(i)}$. Let $\bV = \bV_\nud = \bigoplus_{i\in Q_0} \bV^{(i)}$ be a $Q_0$-graded $\bC$-vector space of dimension $\nud$, that is, $\dim \bV^{(i)} = \nu^{(i)}$ for all $i\in Q_0$. Let $E=E_\nud = \prod_{a\in Q_1} \Hom_\bC(\bV^{(ta)},\bV^{(ha)})$ be the representation space of $Q_0$-graded dimension $\nu$.\footnote{We write linear functions as matrices acting from the left of their arguments, which are column vectors. So the elements of $\Hom_{\bC}(\bV^{ta},\bV^{ha})$ are matrices with $\dim \bV^{ha}$ rows and $\dim \bV^{ta}$ columns.}
 For $i\in Q_0$ let $G^{(i)}=GL(\bV^{(i)})\cong GL_{\nu^{(i)}}$ be the general linear group acting at vertex $i$. Let $G^\bullet = G_\nud = \prod_{i\in Q_0} G^{(i)}$ act  on $E$ by $(g_i\mid i\in Q_0) \cdot (\phi_a\mid a\in Q_1)=(g_{ha} \phi_a g_{ta}^{-1}\mid a\in Q_1)$.

\subsection{Torus weights}\label{SS:torus weights}
Let $T^{Q_1} = (\bC^*)^{Q_1}$ be an algebraic torus with a copy of $\bC^*$ for each arrow $a\in Q_1$. We have $R(T^{Q_1}) \cong \bZ[t_a^{\pm1}\mid a\in Q_1]$ such that the exponential weight of the action of the $a$-th copy of $\C^*$ in $T^{Q_1}$ on $\Hom_{\bC}(\bV^{ta},\bV^{ha})$ is $t_a^{-1}$. We call the $t_a$ \textit{arrow variables}. The action of $T^{Q_1}$ commutes with the action of $G^\bullet$ on $E$. We write $t_{Q_1}$ to refer to the collection of all arrow variables and set $\Z[t_{Q_1}^{\pm 1}]=\bZ[t_a^{\pm1}\mid a\in Q_1]$.

\begin{rem} \label{R:no multiple arrows} If $Q$ has no multiple arrows, that is, for every $(i,j)\in Q_0^2$ there is at most one $b\in Q_1$ such that $tb=i$ and $hb=j$, then we write $t_{i,j}$ for $t_b$ if $b$ is the unique arrow going from $i$ to $j$.
\end{rem}

For each $i\in Q_0$ let $T^{(i)}\subset G^{(i)}$ be the standard maximal torus. Let $T^\bullet=T_{\nu^\bullet}\cong \prod_{i\in Q_0} T^{(i)}$ be the maximal torus in $G^\bullet$. For $i\in Q_0$, let the exponentials of the weights of $T^{(i)}$ be denoted $x^{(i)}_1,x^{(i)}_2,\dotsc,x^{(i)}_{\nu^{(i)}}$.

Let $\mathcal{G} = G^\bullet \times T^{Q_1}$ and $\mathcal{T}=T^\bullet \times T^{Q_1}$. We also write $G_{\nu^{\bullet}}=G^\bullet$ and $\cG_{\nu^\bullet}=\cG$ when it is necessary to distinguish between different dimension vectors.

We use notation such as $X(G)$ and $X_+(G)$ for the integral weights and dominant integral weights. We have a natural identification $X(G^\bullet)=X(GL(\bV))$, since $T^\bullet$ is  a maximal torus in both $GL(\bV)$ and its Levi subgroup $G^\bullet$.

For $\la\in X(\mathcal{G})$ we write $\overline{\la}$ for its image under the forgetful map $X(\mathcal{G}) \to X(G^\bullet)$.

We make the identification $R(\mathcal{T})\cong R(T^\bullet)\otimes_\Z \Z[t_{Q_1}^{\pm 1}]$ and also define its ``completion'' $\hR(\mathcal{T}) = R(T^\bullet)\otimes_\Z \Z((t_{Q_1}))$ where $\Z((t_{Q_1}))$ is the field of formal Laurent series in the variables $t_{Q_1}$. Similarly, we define $\hR(\mathcal{G})= R(G^\bullet)\otimes_{\Z}\Z((t_{Q_1}))$. 

\subsection{Sequences of currents: preview}\label{SS:preview}
The space of quiver symmetric functions $\La^Q = \bigotimes_{i\in Q_0} \La^{(i)}$ is by definition the $Q_0$-fold tensor power of the symmetric function algebra $\La$ over the ring $R(T^{Q_1})$. For a triple $(i,a,\mu)$ with $i\in Q_0$, $a\in \bZ_{>0}$, and $\mu\in\bZ^a$ a weakly decreasing sequence of integers (dominant $GL_a$-weight), we shall define an operator $H^{(i,a)}_\mu$ on $\La^Q$ that we call a \textit{quiver current} (see \eqref{E:H-current}).\footnote{The usual notion of current in the theory of vertex algebras is a generating function of such operators for $i$ fixed, $a=1$, and summing over all integers $\mu$.} The operator $H^{(i,a)}_\mu$ is a quiver analogue of a parabolic Garsia-Jing operator \cite{G} \cite{J} \cite{SZ} (see \eqref{E:H-alternate}).

Consider a sequence of triples
\begin{align}\label{L:the triples}
	(i_1,a_1,\mu(1)),(i_2,a_2,\mu(2)),\dotsc,(i_m,a_m,\mu(m)),
\end{align}
or equivalently a triple of sequences
\begin{align}
\label{E:seq}
	&(\bi,\ba,\mu(\bullet)) \qquad\text{where} \\
	\bi &= (i_1,i_2,\dotsc,i_m) \\
	\ba &= (a_1,a_2,\dotsc,a_m) \\
	\mu(\bullet) &= (\mu(1),\mu(2),\dotsc,\mu(m))
\end{align}
where we emphasize that each $\mu(j)\in X_+(GL_{a_j})$ is a dominant integral weight.

The \textit{quiver Hall-Littlewood symmetric function} $H^{\bi,\ba}_{\mu(\bullet)}\in \La^Q$ is defined by applying a sequence of currents to the vacuum vector $1$:
\begin{align}\label{E:quiver-HL-intro}
	H^{\bi,\ba}_{\mu(\bullet)} = H^{(i_1,a_1)}_{\mu(1)} \dotsm H^{(i_m,a_m)}_{\mu(m)} \cdot 1.
\end{align}
For a $Q_0$-tuple of partitions $\lad\in\bY^{Q_0}$ let $s_\lad = \bigotimes_{i\in Q_0} s_{\la^{(i)}}[X^{(i)}]\in \La^Q$ be the basis of tensor Schur functions.

The \textit{quiver Kostka-Shoji} polynomials $\cK^{\bi,\ba}_{\lad,\mu(\bullet)}(t_{Q_1})$ are the polynomials in the arrow variables defined by the coefficients of the quiver Hall-Littlewood symmetric function at the tensor Schur basis:
\begin{align}\label{E:Kostka Shoji}
	H^{\bi,\ba}_{\mu(\bullet)} = \sum_{\lad} \cK^{\bi,\ba}_{\lad,\mu(\bullet)}(t_{Q_1}) s_{\lad}.
\end{align}

\subsection{Indexing}
Two forms of indexing will be used for almost all objects. \textit{Sequence notation} comes directly from the sequences \eqref{E:seq} and uses lowered parenthesis notation, e.g., $\mu(k)$ for $1\le k\le m$. The other is \textit{vertex notation}, based on grouping terms in these sequences according to the vertex. Vertex notation uses parenthesized superscripts, e.g., $G^{(i)}$ for $i\in Q_0$.

\subsection{Standard big partial flag}\label{SS:big partial flag}
The pair $(\bi,\ba)$ defines a sequence of dimension vectors and a standard ``big partial flag". Starting with the zero dimension vector,
add dimension $a_m$ at vertex $i_m$. Then then add dimension $a_{m-1}$ at vertex $i_{m-1}$, and so on.\footnote{This is consistent with the order in which the operators are applied in \eqref{E:quiver-HL-intro}.} That is, we consider the dimension vectors $0$, $a_m f^{(i_m)}$, $a_m f^{(i_m)} + a_{m-1} f^{(i_{m-1})}$, etc. The final dimension vector is denoted
\begin{align}\label{E:nu of i a}
\nud &= \nud(\bi,\ba) = \sum_{k=1}^m a_k f^{(i_k)}\qquad\text{or equivalently} \\
\nu^{(i)} &= \sum_{\substack{ k \\ i_k=i}} a_k.
\end{align}
We define $a^{(i)}$ to be the subsequence of $\ba$ consisting of the $a_k$ such that $i_k=i$.

\begin{ex} \label{X:ia}
Let $Q_0=\{0,1\}$ and $Q_1 = \{(0,0),(0,1)\}$. Let $(\bi,\ba)$ be given by
\begin{align*}
	\begin{array}{|c||c|c|c|c|c|} \hline
		k                & 1 & 2 & 3 & 4 & 5 \\ \hline \hline
		i_k              & 0 & 0 & 1 & 0 & 1  \\ \hline
		a_k              & 1 & 1 & 1 & 1 & 2 \\ \hline
	\end{array}
\end{align*}
We have $a^{(0)} = (1,1,1)$, $\nu^{(0)}=1+1+1=3$, $a^{(1)}=(1,2)$, and $\nu^{(1)}=1+2=3$.
\end{ex}

{}From now on fix $(\bi,\ba)$ and $\nud=\nud(\bi,\ba)$.

For $i\in Q_0$ let $\cB^{(i)} =\{e^{(i)}_1,\dotsc,e^{(i)}_{\nu^{(i)}}\}$  be a fixed $T^{(i)}$-weight basis of $\bV^{(i)}$
$$
\bV^{(i)} = \bC e^{(i)}_1 \oplus \bC e^{(i)}_2 \oplus \dotsm \oplus \bC e^{(i)}_{\nu^{(i)}}.
$$
with corresponding set of exponential weights  $x^{(i)}=(x_1^{(i)},x_2^{(i)},\dotsc,x_{\nu^{(i)}}^{(i)})$.

We now define the sequence notation for the above.
For all $1\le k\le m$ let $\cB(k)$ consist of $a_k$ consecutive elements of $\cB^{(i_k)}$, 
such that if $i_k=i_\ell=i$ with $k<\ell$ then the elements of $\cB(k)$ precede those of $\cB(\ell)$ in $\cB^{(i)}$. We denote by 
$e(k)_1, e(k)_2,\dotsc, e(k)_{a_k}$ the elements of $\cB(k)$ and denote by $x(k)$ the set of $x$ variables associated with the basis elements of $\cB(k)$.

\begin{ex} \label{X:standard flag} Continuing the previous example we have $\cB^{(0)} = \{e^{(0)}_1|e^{(0)}_2|e^{(0)}_3\}$, $\cB^{(1)} = \{e^{(1)}_1|e^{(1)}_2,e^{(1)}_3\}$ where the vertical lines break each $\cB^{(i)}$ into the subsets $\cB(k)$ for which $i_k=i$.
\begin{align*}
\begin{array}{|c||c|c|c|c|c|}\hline
k      & 1 & 2 & 3 & 4 & 5 \\ \hline
& e(1)_1 & e(2)_1 & e(3)_1 & e(4)_1 & e(5)_1, e(5)_2 \\ \hline
\cB(k) & e^{(0)}_1 & e^{(0)}_2 & e^{(1)}_1 & e^{(0)}_3 & e^{(1)}_2,e^{(1)}_3 \\ \hline
x(k) & x^{(0)}_1 & x^{(0)}_2 & x^{(1)}_1 & x^{(0)}_3 & x^{(1)}_2,x^{(1)}_3 \\ \hline
\end{array}
\end{align*}
\end{ex}

Define $\bV(k)$ to be the $Q_0$-graded subspace of $\bV$ whose basis is $\bigsqcup_{\ell>k} \cB(\ell)$ for $1\le k\le m$. We have
\begin{align*}
	\bV = \bV(0) \supset \bV(1) \supset\dotsm\supset \bV(m-1)\supset \bV(m)=0.
\end{align*}
We call $V(\bullet)$ the \textit{standard big partial flag}.

\subsection{Variety $\Fl_{\bi,\ba}$ of big partial flags}
A flag of type $(\bi,\ba)$ is a decreasing sequence $\bF(\bullet)$ of $Q_0$-graded vector subspaces
\begin{align*}
	\bV = \bF(0) \supset \bF(1) \supset \bF(2)\dotsm \supset \bF(m)=0
\end{align*}
such that for all $1\le k\le m$:
\begin{align} \label{E:big flag}
	\dim (\bF(k)^{(i)}/\bF(k-1)^{(i)}) = \begin{cases}
		a_k & \text{if $i_k=i$} \\
		0 & \text{otherwise.}
		\end{cases}
\end{align}
Let $\Fl_{\bi,\ba}$ be the variety of flags of type $(\bi,\ba)$.
It has basepoint $\bV(\bullet)$.

\subsection{Multi-partial flags} \label{SS:multi partial flags} 
The projections $\bV\to \bV^{(i)}$ for $i\in Q_0$ induce a $G^\bullet$-equivariant isomorphism
\begin{align*}
	\Fl_{\bi,\ba} \overset{\sim}{\longrightarrow} \prod_{i\in Q_0} \Fl_{a^{(i)}}(\bV^{(i)})
\end{align*}
where $\Fl_{a^{(i)}}(\bV^{(i)})$ is a variety of partial flags in $\bV^{(i)}$ of dimension jumps given by the sequence $a^{(i)}$.

\subsection{Lusztig's iterated convolution diagram \cite[\S1.5]{L:conv}}
Let $\phi \in E$. Say that a flag $\bF(\bullet)\in \Fl_{\bi,\ba}$ is $\phi$-stable if
\begin{align}
	\phi_a(\bF(k)^{(ta)}) \subset \bF(k)^{(ha)}\qquad\text{for all $a\in Q_1$, $1\le k\le m$.}
\end{align}
Say that $\bF(\bullet)$ is strictly $\phi$-stable if
\begin{align}
	\phi_a(\bF(k-1)^{(ta)}) \subset \bF(k)^{(ha)}\qquad\text{for all $a\in Q_1$, $1\le k\le m$.}
\end{align}
Say that $\phi$ is \textit{nilpotent} if there is an $N$ such that for any directed path $b_N\dotsm b_2b_1$\footnote{This ugly contravariance is a consequence of writing linear functions on the left of their arguments.} in $Q$ with $b_j\in Q_1$, the composition $\phi_{b_1}\phi_{b_2}\dotsm\phi_{b_N}$ is the zero map. 

\begin{lem} \cite[Lemma 1.8]{L:conv}. Let $\phi\in E$.
\begin{enumerate}
	\item[(a)] For quivers without loops, for $\bF(\bullet)\in\Fl_{\bi,\ba}$, $\bF(\bullet)$ is $\phi$-stable if and only if it is strictly $\phi$-stable. Moreover $\phi$ is nilpotent if and only if there is a $\phi$-stable flag $\bF(\bullet)\in \Fl_{\bi,\ba}$.
	\item[(b)] For quivers with loops, if there is a strictly $\phi$-stable flag $\bF(\bullet)\in\Fl_{\bi,\ba}$ then $\phi$ is nilpotent.
\end{enumerate}
\end{lem}


Let $\cW=\cW_{\bi,\ba} \subset \Fl_{\bi,\ba} \times E$ be the variety of pairs $(\bF(\bullet), \phi)$ such that $\bF(\bullet)$ is strictly $\phi$-stable. For quivers without loops this is Lusztig's convolution diagram \cite[\S 1.5]{L:conv}.

The projection 
\begin{align}\label{E:p}
	\cW \overset{p}{\longrightarrow} \Fl_{\bi,\ba}
\end{align}
gives $\cW$ the structure of a $G^\bullet$-homogeneous vector bundle \cite[Lemma 1.6]{L:conv}.
The other projection map $\mathrm{Spr}=\mathrm{Spr}_{\bi,\ba}:\cW\to E$ is proper; it is known to give a desingularization of its image when $Q$ is a Dynkin quiver \cite{R} or a cyclic quiver ~\cite{ADK}~\cite{Sch}.

\begin{ex} \label{X:nilpotent adjoint orbit}
For the single loop quiver, $\bi$ is unnecessary, the dimension vector is just a dimension, the sequence $\ba$ gives the diagonal block sizes for a standard parabolic $P \subset G$, $\cW=T^*(G/P)$, and $\Spr$ is the Springer desingularization of the closure of an adjoint orbit of a nilpotent. If all $a_k$ are $1$ then the nilpotent is principal. If $m=1$ (there is only one step in the sequence $\ba$) then the nilpotent is zero.
\end{ex}

\begin{ex} \label{X:cycle}
For the directed cycle quiver $Q_0 = \bZ/r\bZ$, $Q_1=\{(i,i+1) \mid i\in Q_0\}$ with $r\ge 2$ and a positive integer $n$, define $(\bi,\ba)$ by $m=rn$, $\bi=(0,1,2,\dotsc)$ where vertices are considered modulo $r$, and $a_k=1$ for $1\le k\le m$. Then $Z_{\bi,\ba}$ is the vector bundle of Finkelberg and Ionov \cite{FI} (see also Example~\ref{X:cyclic vanishing} below).
\end{ex}

\begin{rem}
The generality afforded by Lusztig's convolution diagrams $\cW_{\bi,\ba}$ is far greater than that of our earlier (unpublished) work \cite{OS:unpublished} which initiated the study of quiver Hall-Littlewood functions and Kostka-Shoji polynomials in a more restricted setting. This involved a total order $i_1<\dotsm<i_r$ on $Q_0$ compatible with an acyclic subquiver $\hat{Q}$ of $Q$. This is recovered as a special case by choosing the data $\bi=(i_1,\dotsc,i_r,i_1,\dotsc,i_r,\dotsc)$, with each vertex appearing the same number of times, and $\ba=(1,1,1,\dotsc)$.
\end{rem}

\subsection{Vector bundle weights}\label{SS:bundle-weights}
Let $W\subset E = E_{\nu^\bullet(\bi,\ba)}$ be the fiber of $\cW$ over the basepoint $\bV(\bullet)\in \Fl_{\bi,\ba}$. It carries an action of $P^\bullet = \prod_{i\in Q_0} P^{(i)}$ where $P^{(i)}\subset G^{(i)}$ is the parabolic that stabilizes the projection of the standard big partial flag $\bV(\bullet)$ to $\bV^{(i)}$. We have the partial flag varieties
$G^{(i)}/P^{(i)}\cong \Fl_{a^{(i)}}(\bV^{(i)})$. Note that $P^\bullet$ is {\em lower triangular} with respect to our ordered basis of $\bV$.


$W$ has a $\mathcal{T}$-weight basis $R_{\bi,\ba}$ consisting of 
vectors $\alpha^a_{p,q}(k,\ell)\in W$ for
\begin{enumerate}
	\item each $k < \ell$
	\item each arrow $a \in Q_1$ such that $ha=i_k$ and $ta=i_\ell$
	\item each $p,q$ such that $1\le p\le a_k$ and $1\le q\le a_\ell$.
\end{enumerate}
The vector $\alpha^a_{p,q}(k,\ell)$ has exponential weight
\begin{align}\label{E:root weight}
	\exp(\wt(\alpha^a_{p,q}(k,\ell))) = t_a^{-1} x(\ell)_q / x(k)_p.
\end{align}
We shall use the shorthand
\begin{align}\label{E:t root}
  t_\alpha = t_a \qquad\text{for $\alpha=\alpha^a_{p,q}(k,\ell)\in R_{\bi,\ba}$.}
\end{align}
We have
\begin{align*}
	\cW = G^\bullet \times^{P^\bullet} W.
\end{align*}

\begin{ex}\label{X:bundle} Below is a picture of $R_{\bi,\ba}$ for the data of Example~\ref{X:ia}.
\begin{align*}
\begin{array}{|c||c|c|c|c|cc|}\hline 
i_k \backslash i_\ell  & 0 & 0 & 1 & 0 & \multicolumn{2}{c|}{1} \\ \hline  \hline
0 & & * & * & * & * & *  \\ \hline
0 &  &   & * & * & * & * \\ \hline
1 & &   &   &   &   &   \\ \hline
0 &  &   &   &   & * & * \\ \hline
\multirow{2}{*}{1}&   &   &   &   & & \\ 
 &  &  &   &  & & \\ \hline
\end{array}
\end{align*}
Note that the rows with $i_k=1$ have no roots because the vertex $1$ has no outgoing arrows.
\end{ex}

\begin{ex} \label{X:double arrow} Let $Q_0=\{0,1\}$ and $Q_1=\{a,b\}$ with $ha=hb=0$ and $ta=tb=1$. Let $\bi=(0,1)$ and $\ba=(2,1)$. Then $R_{\bi,\ba}=\{\alpha^a_{1,1}(1,2),\alpha^a_{2,1}(1,2)\alpha^b_{1,1}(1,2),\alpha^b_{2,1}(1,2) \}$ can be depicted by
\begin{align*}
\begin{array}{|c||cc|c|}\hline 
i_k \backslash i_\ell  & \multicolumn{2}{c|}{0}&1 \\ \hline  \hline
\multirow{2}{*}{0}& \hphantom{x}   &  & 2\\ & & &2 \\ \hline
1 & & & \\ \hline
\end{array}
\end{align*}
The entries $2$ indicate that $E$ consists of two linear maps $\C^2\to\C^1$.
\end{ex}

\subsection{Twisting $\cW$ by a vector bundle}\label{SS:vector bundle}
Given $(\bi,\ba)$, consider a sequence of $m$ weights $\mu(\bullet)=(\mu(1),\mu(2),\dotsc,\mu(m))$
where $\mu(k) \in X_+(GL_{a_k})$. For $i\in Q_0$ let $\mu(\bullet)^{(i)}$ denote the subsequence of weights $\mu(k)$ for which $i_k=i$, and let $\mu^{(i)}\in \bZ^{\nu^{(i)}}=X(GL(\bV^{(i)}))$ be the (not necessarily dominant) weight obtained by concatenating the weights in the sequence $\mu(\bullet)^{(i)}$. We will use the notation $\mu^\bullet$ to denote the $Q_0$-tuple of weights $(\mu^{(i)}\mid i\in Q_0)$.

\begin{ex} For the running example (Example~\ref{X:ia}), let $$\mu(\bullet)=((3),(2),(4),(4),(2,1)).$$ Then
$\mu(\bullet)^{(0)} = ((3),(2),(4))$, $\mu^{(0)}=(3,2,4)$,
$\mu(\bullet)^{(1)} = ((4),(2,1))$, and $\mu^{(1)}=(4,2,1)$;
$\mu^\bullet$ is the $Q_0$-tuple $((3,2,4),(4,2,1))$.
\end{ex}

For each $i\in Q_0$ the sequence $\mu(\bullet)^{(i)}$ defines a dominant weight for the standard Levi subgroup $L^{(i)}\subset P^{(i)}\subset G^{(i)}$ with diagonal block sizes given by $a^{(i)}$ (see \S \ref{SS:preview}), where $P^{(i)}$ is defined as in \S\ref{SS:bundle-weights}.
Let $\Fl(\bV^{(i)})$ be the variety of complete flags in $\bV^{(i)}$. Let $\cL_{\mu(\bullet)}$ be the $G^\bullet$-equivariant vector bundle on $\prod_{i\in Q_0} \Fl(\bV^{(i)})$ whose weight at the basepoint of $\Fl(\bV^{(i)})$ is $\mu^{(i)}$. We may identify $\Fl(\bV^{(i)})$ with $G^{(i)}/B^{(i)}_-$, where $B^{(i)}_-$ is the {\em lower triangular} Borel subgroup with respect to the ordered basis $\cB^{(i)}$. Let $B^\bullet_-=\prod_{i\in Q_0} B^{(i)}_-$. Then $\cL_{\mu(\bullet)}=G^\bullet\times^{B^\bullet_-}\C_{\mu^\bullet}$.

Let $r: \prod_{i\in Q_0} \Fl(\bV^{(i)})  \to \prod_{i\in Q_0} \Fl_{a^{(i)}}(\bV^{(i)})$ be the product of projections, which for the $i$-th factor maps the complete flag variety $\Fl(\bV^{(i)})$ to the partial flag variety $\Fl_{a^{(i)}}(\bV^{(i)})$.
Define the vector bundle $\cW^{\mu(\bullet)}$ on $\cW$ by
\begin{align}
	\cW^{\mu(\bullet)} = p^*(r_*(\cL_{\mu(\bullet)})).
\end{align}

\subsection{Quiver Hall-Littlewood series}
We define the {\em quiver Hall-Littlewood series} $\chi_{\mu(\bullet)}^{\bi,\ba}$ to be the $\mathcal{G}$-equivariant Euler characteristic of the global sections functor applied to $\cW^{\mu(\bullet)}$:
\begin{align*}
	\chi_{\mu(\bullet)}^{\bi,\ba} &= \sum_{p\ge0} (-1)^p \mathrm{ch}_{\mathcal{G}} H^p(\cW,\cW^{\mu(\bullet)}).
\end{align*}
We can compute this as follows. Let $J^\bullet$ be the antisymmetrization operator over the Weyl group $S^\bullet = \prod_{i\in Q_0} S^{(i)}$ where $S^{(i)}$ is the symmetric group $S_{\nu^{(i)}}$, the Weyl group for $GL(\bV^{(i)})$:
\begin{align}
	J^\bullet = \sum_{\wbul\in S^\bullet} (-1)^{\wbul} \wbul.
\end{align}
Let $\rho^\bullet$ be the $Q_0$-tuple of weights $\rho^{(i)}=\rho_{\nu^{(i)}}$ where $\rho_n=(n-1,n-2,\dotsc,1,0)$ is a $GL_n$-weight. For $f\in \hR(\mathcal{T})$ define the Demazure operator
\begin{align}\label{E:Dw0}
	D_{w_0^{\bullet}}(f) = J^{\bullet}(x^{\rho^\bullet})^{-1} J^{\bullet}(x^{\rho^\bullet} f).
\end{align}
Let $x^{\mu(\bullet)}$ be the monomial
\begin{align}\label{E:mu(bullet)}
	x^{\mu(\bullet)} &= \prod_{k=1}^m x(k)^{\mu(k)} 
\end{align}
where $x(k)$ is defined in \S \ref{SS:big partial flag}. Finally, let
\begin{align}\label{E:Bia}
	B_{\bi,\ba} &= \mathrm{ch}_{\mathcal{T}} \,\Sym(W^\vee)\\
\notag	&= \prod_{1\le k<\ell\le m} \prod_{\substack{a \in Q_1 \\ ta = i_k \\ ha=i_\ell}} \prod_{\substack{(y,z)\in x(k)\times x(\ell)}}  (1 - t_a y/z  )^{-1}
\end{align}
where $\Sym(W^\vee)$ is the symmetric algebra of the dual of $W$. Then:

\begin{prop} The quiver Hall-Littlewood series is given by:
\begin{align}\label{E:chi-D}
\chi_{\mu(\bullet)}^{\bi,\ba} &= D_{w_0^\bullet} (x^{\mu(\bullet)} B_{\bi,\ba}).
\end{align}
\end{prop}

This is understood as an element of $R(T^\bullet)[[t_{Q_1}]]$, i.e., as the character of a (virtual) $T^{Q_1}$-graded locally finite $G^\bullet$-module.

\begin{proof}
\comment{
For any integral weight $\xi^\bullet\in X(G^\bullet)$, one has the Borel-Weil-Bott theorem:
\begin{align*}
D_{w_0^\bullet}(x^{\xi^\bullet}) = \sum_{p\ge 0}(-1)^p \mathrm{ch}_{\mathcal{G}} H^p(G^\bullet/B^\bullet_-, G^\bullet \times^{B^{\bullet}_-} \C_{\xi^\bullet}).
\end{align*}
When $\xi^\bullet$ is $B^\bullet$-dominant, this is equal to the character of the irreducible $G^\bullet$-representation with highest weight $\xi^\bullet$ with respect to $B^\bullet$.
}

We have a canonical isomorphism
\begin{align*}
H^p(\cW,\cW^{\mu(\bullet)}) &\cong H^p\left(\prod_{i\in Q_0} \Fl_{a^{(i)}}(\bV^{(i)}), r_*\cL_{\mu(\bullet)}\otimes \Sym(\cW^\vee)\right).
\end{align*}
Next we consider the vector bundle $\widetilde{\cW} = G^\bullet\times^{B^\bullet_-} W \cong r^* \cW$ on $\prod_{i\in Q_0}\Fl(\bV^{(i)})$. We have
\begin{align*}
&H^p\left(\prod_{i\in Q_0} \Fl_{a^{(i)}}(\bV^{(i)}), r_*\cL_{\mu(\bullet)}\otimes \Sym(\cW^\vee) \right)\\
&\qquad \cong H^p\left(\prod_{i\in Q_0} \Fl(\bV^{(i)}), \cL_{\mu(\bullet)}\otimes \Sym(\widetilde{\cW}^\vee) \right)
\end{align*}
because $H^q(P^\bullet/B^\bullet_-,\cL_{\mu(\bullet)})=0$ for $q>0$; the latter ensures that $R^q r_*\cL_{\mu(\bullet)}=0$ for $q>0$. Finally, invoking the Borel-Weil-Bott theorem, we have
\begin{align*}
	\chi_{\mu(\bullet)}^{\bi,\ba} &= \mathrm{ch}_{\mathcal{G}} \sum_{p\ge0} (-1)^p H^p\left(\prod_{i\in Q_0}\Fl(\bV^{(i)}), \cL_{\mu(\bullet)} \otimes \Sym(\widetilde{\cW}^\vee)\right) \\
	&= D_{w_0^\bullet} (x^{\mu(\bullet)} B_{\bi,\ba}).\notag\qedhere
\end{align*}
\end{proof}

\begin{ex} Consider the data of Example~\ref{X:ia}. Writing $k,\ell$ to label the groups of factors, we have
\begin{align*}
  B_{\bi,\ba}^{-1} &= 
  \underbrace{\left(1-t_{00} \dfrac{x_1^{(0)}}{x_2^{(0)}} \right)}_{1,2} \underbrace{\left(1-t_{01}\dfrac{x_1^{(0)}}{x_1^{(1)}} \right)}_{1,3}
  \underbrace{\left(1-t_{00} \dfrac{x_1^{(0)}}{x_3^{(0)}} \right)}_{1,4}
  \underbrace{\left(1-t_{01} \dfrac{x_1^{(0)}}{x_2^{(1)}} \right)
  \left(1-t_{01} \dfrac{x_1^{(0)}}{x_3^{(1)}} \right)}_{1,5}
  \\ 
  & \quad\,\underbrace{\left(1-t_{01} \dfrac{x_2^{(0)}}{x_1^{(1)}} \right)}_{2,3}
  \underbrace{\left(1-t_{00} \dfrac{x_2^{(0)}}{x_3^{(0)}} \right)}_{2,4}
 \underbrace{  \left(1-t_{01} \dfrac{x_2^{(0)}}{x_2^{(1)}} \right)
   \left(1-t_{01} \dfrac{x_2^{(0)}}{x_3^{(1)}} \right)}_{2,5} \\
   &\quad\,\underbrace{\left(1-t_{01} \dfrac{x_3^{(0)}}{x_2^{(1)}} \right)\left(1-t_{01} \dfrac{x_3^{(0)}}{x_3^{(1)}} \right)}_{4,5}.
\end{align*}
\end{ex}

\begin{ex} For the data of Example \ref{X:double arrow},
	\begin{align*}
		B_{\bi,\ba}^{-1} = \left(1-t_a \dfrac{x_1^{(0)}}{x_1^{(1)}} \right)
		\left(1-t_a \dfrac{x_2^{(0)}}{x_1^{(1)}} \right)
		\left(1-t_b \dfrac{x_1^{(0)}}{x_1^{(1)}} \right)
		\left(1-t_b \dfrac{x_2^{(0)}}{x_1^{(1)}} \right).
	\end{align*}
\end{ex}

\begin{ex} \label{X:nullcone sl2}
The series $\chi_{\mu(\bullet)}^{\bi,\ba}$ is generally infinite.
For the single loop quiver at $Q_0=\{0\}$ with $\bi=(0,0)$, $\ba=(1,1)$, and $\mu(\bullet)=((0),(0))$ (the graded character of the nullcone in $\mathfrak{gl}_2$) we have
\begin{align*}
	\chi_{\mu(\bullet)}^{\bi,\ba} = \sum_{r\ge0} t_{0,0}^r (x_1^{(0)} x_2^{(0)})^{-r} s_{(2r,0)}(x_1^{(0)},x_2^{(0)}).
\end{align*}
In contrast, the quiver Hall-Littlewood symmetric function --- defined in \eqref{E:quiver HL} below --- is always finite. It is the polynomial truncation of the quiver Hall-Littlewood series (cf. Theorem~\ref{T:quiver HL functions and Kostka}). In this example we have $H^{\bi,\ba}_{\mu(\bullet)} = s_\varnothing[X^{(0)}]$.
\end{ex}

\subsection{Quiver Kostka-Shoji polynomials}\label{SS:Kostka}
Let $\lad\in X_+^\bullet = \prod_{i\in Q_0} X_+(GL(\bV^{(i)}))$. The \textit{quiver Kostka-Shoji polynomials} $\cK_{\lad,\mu(\bullet)}^{\bi,\ba}(t_{Q_1})\in \bZ[t_{Q_1}]$ are defined by the expansion\footnote{The connection to \eqref{E:Kostka Shoji} is given by Theorem~\ref{T:quiver HL functions and Kostka} below.}
\begin{align}\label{E:defquiverKostka}
	\chi_{\mu(\bullet)}^{\bi,\ba} = \sum_{\lad\in X_+^\bullet} \cK^{\bi,\ba}_{\lad,\mu(\bullet)}(t_{Q_1}) s_\lad(x)
\end{align}
of the quiver Hall-Littlewood series into products of Schur polynomials $s_\lad(x)=\prod_{i\in Q_0}s_{\la^{(i)}}(x^{(i)}_1,\dotsc,x^{(i)}_{\nu^{(i)}})$. The fact that $\cK_{\lad,\mu(\bullet)}^{\bi,\ba}(t_{Q_1})\in \bZ[t_{Q_1}]$ is justified by the Kostant partition formula \eqref{E:Kostant} below.

We use the notation $[f] g$ for the coefficient of $f$ in $g$. We have
\begin{align*}
	\cK_{\lad,\mu(\bullet)}^{\bi,\ba}(t_{Q_1}) &= [s_\lad(x)] D_{w_0^\bullet} x^{\mu^\bullet} B_{\bi,\ba} \\
	&= [s_\lad(x)] J^\bullet(x^{\rho^\bullet})^{-1} J^\bullet(x^{\mu(\bullet)+\rhod}) B_{\bi,\ba} \\
	&= [J^\bullet(x^{\lad+\rhod})] J^\bullet(x^{\mu^\bullet+\rhod}) B_{\bi,\ba} \\
	&= [x^{\lad+\rhod}] J^\bullet(x^{\mu^\bullet+\rhod}) B_{\bi,\ba} \\
	&= \sum_{\wbul \in S^\bullet} (-1)^{\wbul} [x^{\lad+\rhod-\wbul(\mu^\bullet+\rhod)}] w (B_{\bi,\ba}) \\
	&= \sum_{\wbul\in S^\bullet} (-1)^{\wbul} [x^{\wbul^{-1}(\lad+\rhod)-(\mu^\bullet+\rhod)}] B_{\bi,\ba}.
\end{align*}
Therefore
\begin{align}\label{E:Kostant}
	\cK_{\lad,\mu(\bullet)}^{\bi,\ba}(t_{Q_1}) &= \sum_{\wbul \in S^\bullet}
	(-1)^{\wbul} \sum_{m:R_{\bi,\ba}\to \bZ_{\ge0}} \prod_{\alpha\in R_{\bi,\ba}} t_\alpha^{m(\alpha)}
\end{align}
where the Kostant partition $m$ satisfies 
\begin{align}\label{E:Kostantweight}
\sum_{\alpha\in R_{\bi,\ba}} m(\alpha) \overline{\alpha} = (\wbul)^{-1}(\lad+\rhod)-(\mu^\bullet+\rhod).
\end{align}
See \textsection \ref{SS:torus weights} for the meaning of $\overline{\alpha}$.

Say that $\mu(\bullet)$ is $(\bi,\ba)$-dominant if, for each $i\in Q_0$, the concatenated weight $\mu^{(i)}$ (see \S \ref{SS:vector bundle}) is dominant (weakly decreasing).

\begin{conj} \label{conj:positivity}
Suppose $\mu(\bullet)$ is $(\bi,\ba)$-dominant. Then for any $p>0$, $$H^p(\cW,\cW^{\mu(\bullet)})=0. $$ Hence for any $\lad\in X_+^\bullet$ the quiver Kostka-Shoji polynomial $\cK_{\lad,\mu(\bullet)}^{\bi,\ba}(t_{Q_1})$ has  nonnegative integer coefficients.
\end{conj}

\begin{ex} \label{X:loop vanishing}
For the single loop quiver, the quiver Kostka-Shoji polynomials are Kostka-Foulkes polynomials \cite{Mac} when $\ba=(1,1,\dotsc)$. Higher cohomology vanishing is an instance of \cite[Theorem 2.4]{B}. For general $\ba$, the quiver Kostka-Shoji polynomials are parabolic Kostka polynomials \cite{Bro} \cite{SW}. Higher vanishing was proposed in \cite{Bro} and is still open.
\end{ex}

\begin{ex} \label{X:cyclic vanishing} For cyclic quivers, the quiver Kostka-Shoji polynomials were first introduced and studied by Finkelberg and Ionov \cite{FI} in the setting of Example~\ref{X:cycle}.
It is conjectured in \cite{FI} and proved in \cite{Sh3} that when every arrow parameter is set to a single parameter $t$, these quiver Kostka-Shoji polynomials recover Shoji's polynomials $K^-_{\lad,\mu^\bullet}(t)$ \cite{Sh2}. 
In this setting, Conjecture~\ref{conj:positivity} is an immediate consequence of \cite{P}, as explained in \cite{FI}.
\end{ex}

\begin{rem}\label{R:vanish}
In general, Conjecture~\ref{conj:positivity} is known for sufficiently regular $(\bi,\ba)$-dominant $\mu(\bullet)$, again by a result of Panyushev \cite{P}; see \cite{H}.
\end{rem}

\subsection{Dominance}\label{SS:dom}

For $\lad,\mud \in X^\bullet = X(G^\bullet) = \bigoplus_{i\in Q_0} X(GL(\bV^{(i)}))$, say that $\lad\dom\mud$ if $$\lad-\mud \in \sum_{\alpha\in R_{\bi,\ba}} \Z_{\ge0}\,  \overline{\alpha} + \sum_{\alpha\in R_+(G^\bullet)} \Z_{\ge0}\, \alpha$$ where $\overline{\alpha}$ is defined in \textsection \ref{SS:torus weights}. Any $\dom$-relation is a relation in the usual dominance order on $X(GL(\bV))$; the latter uses all positive roots for $GL(\bV)$ while $\dom$ uses only the positive roots for the Levi subgroup $G^\bullet$ and those in the projection to $X(T^\bullet)$ of the roots in $R_{\bi,\ba}$.

Let $\mu^{\bullet}\in X^\bullet$ be derived from $(\bi,\ba,\mu(\bullet))$ as in \S \ref{SS:vector bundle}. 

\begin{lem} \label{L:dom}
$\cK_{\lad,\mu(\bullet)}^{\bi,\ba}(t_{Q_1})$ is zero unless $\lad\dom\mu^{\bullet}$.
\end{lem}
\begin{proof}
This is immediate from \eqref{E:Kostant}.
\end{proof}

\subsection{Cycles}
The nontrivial portion of the grading is encoded by directed cycles. 

\begin{lem} \label{L:cycle grading}
Every polynomial $\cK^{\bi,\ba}_{\lad,\mu(\bullet)}(t_{Q_1})$ is a single monomial times a Laurent polynomial with integer coefficients, in products of arrow variables coming from directed cycles.
\end{lem}
\begin{proof}
We have $X(GL(V^{(i)}))\cong \Z^{\nu^{(i)}}$ for $i\in Q_0$.
Let $\pr:X^\bullet \to \Z^{Q_0}$ be the linear map sending 
$(a_1,\dotsc,a_{\nu^{(i)}})\in X(GL(V^{(i)}))$ to
$(a_1+a_2+\dotsm+a_{\nu^{(i)}}) f^{(i)}$ with $f^{(i)}\in \Z^{Q_0}$ as defined in \textsection \ref{SS:data}. Letting $S^\bullet$ act on $\Z^{Q_0}$ by the identity, $\pr$ is $S^\bullet$-equivariant. By \eqref{E:Kostant} and \eqref{E:Kostantweight} the weight of every Kostant partition in \eqref{E:Kostant} is sent by $\pr$ to the same vector, namely, $\pr(\lad-\mu^\bullet)$. Recalling $\overline{\alpha}$ from \textsection \ref{SS:torus weights}, the kernel of the restriction of $\pr$ to $\sum_{\alpha\in R_{\bi,\ba}} \Z \overline{\alpha}$, is generated by sums of collections of vectors $\alpha^b_{p,q}(k,\ell)\in R_{\bi,\ba}$ whose edges $b$ form directed cycles in $Q$.
\end{proof}

\begin{rem} This uses the general fact that circulations in directed graphs are generated by directed cycles.
\end{rem}

\begin{rem}
Lemma~\ref{L:cycle grading} is consistent with the number of dilation symmetries acting on the equivariant $K$-groups of Nakajima varieties, namely the rank of $H_*(Q)$ where the quiver $Q$ is regarded as a topological space. (We thank Michael Finkelberg for this clarifying remark.)
\end{rem}

\section{Quiver Hall-Littlewood symmetric functions via creation operators}

In this section we define the quiver Hall-Littlewood symmetric functions and show that they are lifts of the quiver Hall-Littlewood series.

\subsection{Symmetric functions}

Let $\La$ denote the algebra of symmetric functions with coefficients in the ring $R(T^{Q_1})=\Z[t_a^{\pm1}\mid a\in Q_1]$ of Laurent polynomials in the arrow variables. We freely use standard plethystic notation from the theory of symmetric functions, mostly following the notation of \cite{LR}. For instance, the projection from $\Lambda$ to symmetric polynomials in finitely many variables $x_1,x_2,\dotsc,x_n$ is denoted $f\mapsto f[x_1+x_2+\dotsm+x_n]$.

The element $\Omega=\sum_{k\ge 0} h_k=\exp(\sum_{r>0} p_r/r)$, which belongs to a formal completion of $\Lambda$, will play an important role. Here $h_k$ is the $k^{\mathrm{th}}$ complete homogeneous symmetric function and $p_r$ is the $r^{\mathrm{th}}$ power sum. Many of our formulas can be understood using formal properties of $\Omega$, such as:
\begin{align}
\Omega[X+Y]&=\Omega[X]\Omega[Y]\\
\Omega[-X] &=1/\Omega[X]\\
\Omega[u] &= 1/(1-u)
\end{align}
where $X,Y$ are alphabets and $u$ is a single variable.

We denote by $\bY$ the set of integer partitions.

\subsection{Quiver symmetric functions}\label{SS:quiver-sym-fun}
Let $\La^Q = \bigotimes_{i\in Q_0} \La^{(i)}$ be the space of quiver symmetric functions, the tensor product of copies of the symmetric function algebra $\La$, one copy per $i\in Q_0$. We use $X^{(i)}$ for the variables at vertex $i$.

Let $i\in Q_0$, $a\in\bZ_{>0}$, and $\mu\in\bY$ a partition with $a$ parts, some of which may be zero. Define the current $H^{(i,a)}_\mu \in \End(\La^Q)$ via the following generating function. Let $U = (u_1,\dotsc,u_a)$ be a set of auxiliary variables. Let $R(U) = u^{-\rho_a} J(u^{\rho_a})= \prod_{1\le i<j\le a} (1-u_j/u_i)$ and $U^* = \sum_{j=1}^a u_j^{-1}$. Let
\begin{align}\label{E:Out}
\Out(i) = \{b\in Q_1 \mid tb = i\}
\end{align}
be the set of arrows coming out of $i$. We define
\begin{align}\label{E:H-current}
	H^{(i,a)}(U)&=\sum_{\beta\in\bZ^a} u^\beta H^{(i,a)}_{\beta} \\
	&= R(U) \Omega[UX^{(i)}] \Omega[-U^*X^{(i)}]^\perp 
    \prod_{b\in \Out(i)}
    \Omega[t_b U^*X^{(hb)}]^\perp\notag
\end{align}
where $f^\perp$ denotes the adjoint to multiplication by $f\in \La^Q$ with respect to the Hall scalar product $\pair{\cdot}{\cdot}$.

Alternatively, consider the generating function of Bernstein operators
\begin{align}\label{E:Bernstein}
	\sum_{m\in\Z} S_m u^m = S(u) = \Omega[uX] \Omega[-u^{-1}X]^\perp.
\end{align}
These operators create Schur functions: $S_{\lambda_1} S_{\lambda_2}\dotsm S_{\lambda_a}\cdot 1 = s_\lambda$ for any partition $\lambda$ with at most $a$ rows; when $\lambda$ is an arbitrary finite sequence of integers we take the left-hand side as the definition of $s_\lambda$. We note that $S_{p-1}S_{q+1}=-S_qS_p$ for all $p,q\in\Z$. We write $S^{(i)}(u)=\sum_m S^{(i)}_m u^m$ for the Bernstein operators with $X$ replaced by $X^{(i)}$. Then
\begin{align}\label{E:H-alternate}
	H^{(i,a)}(U)&=S^{(i)}(u_1)S^{(i)}(u_2)\dotsm S^{(i)}(u_a) 
	\prod_{b\in\Out(i)}
	\Omega[t_b U^*X^{(hb)}]^\perp.
\end{align}

Let $(\bi,\ba,\mu(\bullet))$ index a sequence of currents. 
For $1\le k\le m$ let $u(k)$ be an $a_k$-tuple of auxiliary variables. For $i\in Q_0$ let $u^{(i)}$ be the
ordered union of the $u(k)$ such that $i_k=i$, analogously to the definition of $\mu^{(i)}$ based on $\mu(\bullet)$ as in \textsection \ref{SS:vector bundle}.

The \textit{quiver Hall-Littlewood symmetric function} $H^{\bi,\ba}_{\mu(\bullet)}\in \La^Q$ is defined by
\begin{align}\label{E:quiver HL}
H^{\bi,\ba}_{\mu(\bullet)} = H^{(i_1,a_1)}_{\mu(1)} H^{(i_2,a_2)}_{\mu(2)}\dotsm H^{(i_m,a_m)}_{\mu(m)} \cdot 1.
\end{align}

\begin{thm} \label{T:quiver HL functions and Kostka} 
We have
\begin{align}
  H^{\bi,\ba}_{\mu(\bullet)} = \sum_{\lad \in \bY^\bullet_{\nu^\bullet}} \cK^{\bi,\ba}_{\lad,\mu(\bullet)}(t_{Q_1}) s_{\lad}[\Xd]
\end{align}
where $\bY^{Q_0}_{\nu^\bullet}$ is the set of $Q_0$-tuples of partitions $\lad$ such that $\la^{(i)}$ has at most $\nu^{(i)}$ rows for $i\in Q_0$.
\end{thm}

This sum is finite since the coefficient is zero unless
$\sum_{k=1}^m |\mu(k)| = \sum_{i\in Q_0} |\la^{(i)}|$.

\begin{rem} Theorem \ref{T:quiver HL functions and Kostka} says that the tensor Schur coefficients of $H^{\bi,\ba}_{\mu(\bullet)}$ are special cases of the coefficients of $\chi^{\bi,\ba}_{\mu(\bullet)}$ at irreducible characters, namely, when each $\la^{(i)}$ is a polynomial dominant weight. Conversely, each coefficient of the generally infinite series $\chi^{\bi,\ba}_{\mu(\bullet)}$ occurs as a coefficient of some quiver HL symmetric function, but the sequence of weights must be shifted. Consider Example \ref{X:nullcone sl2}. The general coefficient of $\chi^{\bi,\ba}_{(0),(0)}$ is $\cK^{\bi,\ba}_{(p,-p),((0),(0))} = t_{0,0}^p$ for $p \ge 0$. The quiver Hall-Littlewood symmetric function $H^{\bi,\ba}_{((0),(0))}$ only sees the coefficient for $p=0$. However if we shift the weights in $\mu(\bullet)$ by adding $p$ to every part, by Theorem \ref{T:quiver HL functions and Kostka} we have
\begin{align*}
	\chi^{\bi,\ba}_{((p),(p))} &= (x_1^{(0)}x_2^{(0)})^p \chi^{\bi,\ba}_{((0),(0))} \\
	\cK^{\bi,\ba}_{(p,-p),((0),(0))} &= \cK^{\bi,\ba}_{(2p,0),((p),(p))} \\
	&= \pair{H^{\bi,\ba}_{((p),(p))}}{s_{(2p,0)}},
\end{align*}
which realizes a general coefficient of the quiver Hall-Littlewood series, as a coefficient of some quiver Hall-Littlewood symmetric function.
\end{rem}

\begin{proof}[Proof of Theorem \ref{T:quiver HL functions and Kostka}]
We must commute all skewing operators to the right past all multiplication operators using
\begin{align*}
	\Omega[ZX^{(i)}]^\perp \Omega[X^{(j)}Y] = (\Omega[ZX^{(i)}])^{\delta_{ij}} \Omega[X^{(j)}Y] \Omega[ZX^{(i)}]^\perp
\end{align*}
where $\perp$ is taken with respect to $X^{(i)}$ and $Y$ and $Z$ are auxiliary variables.

The pairs of skewing and multiplication operators which contribute factors are (for $1\le k < \ell \le m$):
\begin{itemize}
\item $\Omega[-u(k)^*X^{(i_k)}]^\perp$ with $\Omega[u(\ell)X^{(i_\ell)}]$ where $i_k=i_\ell$, giving $\Omega[-u(k)^*u(\ell)]$.
\item $\Omega[t_b u(k)^* X^{(i_\ell)}]^\perp$ with $\Omega[u(\ell)X^{(i_\ell)}]$ where $i_k \overset{b}\longrightarrow i_\ell$ (that is, $tb=i_k$ and $hb=i_\ell$), yielding the factor $\Omega[t_b u(k)^* u(\ell)]$.
\end{itemize}
The skewing operators $\Omega[ZX^{(i)}]^\perp$ send $1$ to $1$, so when they reach the right side and act on $1$, they disappear. Therefore we have
\begin{align*}
&\quad\,\,\,H^{(i_1,a_1)}(u(1)) \dotsm H^{(i_m,a_m)}(u(m)) \cdot 1 \\
&= \prod_{1\le k\le m} R(u(k)) \,\Omega[u(k) X^{(i_k)}] 
\prod_{\substack{1\le k<\ell\le m \\ i_k=i_\ell}} \Omega[-u(k)^* u(\ell)] \prod_{\substack{1\le k<\ell\le m \\ b\in Q_1 \\ i_k \overset{b}\rightarrow i_\ell  }} \Omega[t_b u(k)^* u(\ell)] \\
&= \prod_{i\in Q_0} R(u^{(i)}) \Omega[u^{(i)}X^{(i)}]\, B_{\bi,\ba}(u^*)
\end{align*}
where $B_{\bi,\ba}(u^*)$ is the expression \eqref{E:Bia} except the $x$ variables are replaced by the inverses of the auxiliary $u$ variables.

The above expression is a series in the auxiliary variables with coefficients in $\La^Q$, which has the tensor Schur symmetric function basis $s_\lad[X^\bullet] = \prod_{i\in Q_0} s_{\la^{(i)}}[X^{(i)}]$ where $\lad\in \bY^{Q_0}$ is a $Q_0$-tuple of partitions. By the Cauchy formula
\begin{align*}
	\Omega[u^{(i)}X^{(i)}] = \sum_{\substack{\la\in\bY \\ \ell(\la) \le \nu^{(i)}}} s_\la[u^{(i)}] s_\la[X^{(i)}].
\end{align*}
Using \eqref{E:Kostant} we have
\begin{align*}
&\pair{H^{(i_1)}_{\mu(1)} \dotsm H^{(i_m)}_{\mu(m)} \cdot 1}{s_\lad[X^\bullet]}  \\
	&= [u^{\mu^\bullet}] u^{-\rho^\bullet} J^\bullet(u^{\rho^\bullet})
	s_\lad[u^\bullet] \prod_{\substack{1\le k <\ell \le m \\  i_k \overset{b}\rightarrow i_\ell}} \Omega[t_b u(k)^* u(\ell)] \\
	&= [u^{\mud+\rhod}]  J^\bullet(u^{\lad+\rhod})
	 B_{\bi,\ba}(u^*) \\
	 &= [u^{\mud+\rhod}]\sum_{\wbul\in S^\bullet} (-1)^{\wbul}  u^{\wbul(\lad+\rhod)} B_{\bi,\ba}(u^*) \\
	 &= \sum_{\wbul\in S^\bullet} (-1)^{\wbul} [ u^{(\mud+\rhod)-\wbul(\lad+\rhod)}] B_{\bi,\ba}(u^*) \\
	 &= \sum_{\wbul\in S^\bullet} (-1)^{\wbul} [ u^{\wbul(\lad+\rhod)-(\mud+\rhod)}] B_{\bi,\ba}(u) \\
	 &= \cK_{\lad,\mu(\bullet)}(t_{Q_1}).\qedhere
\end{align*}
\end{proof}
\section{Combinatorics}
\subsection{Recurrence}
\label{SS:Morris}
Let $(\bi,\ba,\mu(\bullet))$ be a sequence of $m$ currents,
$(\hbi,\hba,\hmu(\bullet))$ the above data with the first current removed, $U$ an auxiliary alphabet of size $a=a_1$ and $\rho=\rho_a$. For $\sigma\in\Z^a$ define the multi-Bernstein operator $B_\sigma$ by
\begin{align*}
	S(u_1) S(u_2)\dotsm S(u_a) &= 
	\sum_{\sigma\in\Z^a} u^\sigma B_\sigma = u^{-\rho} J(u^\rho) \Omega[U X] \Omega[-U^* X]^\perp.
\end{align*}
Since the following expression is alternating in $U$ we have
\begin{align*}
	J(u^\rho) \Omega[U X] \Omega[-U^* X]^\perp &= 
	\sum_{\sigma\in\Z^a} u^{\sigma+\rho} B_\sigma \\
	&= \sum_{\tau\in X_+(GL_a)} J(u^{\tau+\rho}) B_\tau \\
	&=\sum_{\tau\in X_+(GL_a)} J(u^\rho) s_\tau[U] B_\tau.
\end{align*}

Write $B^{(i)}(U)$ and $B^{(i)}_\tau$ for $B(U)$ and $B_\tau$ with $X^{(i)}$ replacing $X$.

Consider the tuples of partitions  (see \eqref{E:Out}) $\beta^\bullet\in\bY^{\Out(i)}=(\beta^{(b)}\mid b\in \Out(i))$. 
Let $t^\beta = \prod_{b\in\Out(i)} t_b^{|\beta^{(b)}|}$. We have
\begin{align*}
	H^{(i,a)}(U) &= B^{(i)}(U) \prod_{b\in \Out(i)} \Omega[t_b U^* X^{(hb)}]^\perp \\
	&= u^{-\rho} J(u^\rho) \sum_{\tau\in X_+(GL_a)} s_\tau[U] B_\tau^{(i)} 
	\sum_{\beta^\bullet\in\bY^{\Out(i)}} t^\beta \prod_{b\in\Out(i)} s_{\beta^{(b)}}[U^*]  s_{\beta^{(b)}}[X^{(hb)}]^\perp  \\
	&= u^{-\rho} J(u^\rho) \sum_{\substack{\tau,\eta\in X_+(GL_a)}} s_\eta[U] \sum_{\beta^\bullet\in\bY^{\Out(i)}} t^\beta 
	\pair{\eta}{\tau \bigotimes_{b\in\Out(i)} \beta^{(b)*}}_{GL_a} \\ &\qquad B_\tau^{(i)} \prod_{b\in\Out(i)}
	s_{\beta^{(b)}}[X^{(hb)}]^\perp \\
	&= u^{-\rho}  \sum_{\substack{\tau,\eta\in X_+(GL_a)}} J(u^{\eta+\rho}) \sum_{\beta^\bullet\in\bY^{\Out(i)}} t^\beta 
	\pair{\eta\otimes \bigotimes_{b\in\Out(i)} \beta^{(b)}}{\tau}_{GL_a}\\ &\qquad B_\tau^{(i)}\prod_{b\in\Out(i)}
	s_{\beta^{(b)}}[X^{(hb)}]^\perp 
\end{align*}
and
\begin{align*}
	H^{(i,a)}_{\mu(1)} &= [u^{\mu(1)}] u^{-\rho}  \sum_{\substack{\tau,\eta\in X_+(GL_a) \\ \beta^\bullet\in\bY^{\Out(i)}}} t^\beta  J(u^{\eta+\rho}) 
	\pair{\eta\otimes\bigotimes_{b\in\Out(i)} \beta^{(b)}}{\tau }_{GL_a} \\ &\qquad B^{(i)}_\tau 
	\prod_{b\in\Out(i)} s_{\beta^{(b)}}[X^{(hb)}]^\perp \\
	&= \sum_{\substack{\tau\in X_+(GL_a) \\ \beta^\bullet\in\bY^{\Out(i)}}} t^\beta 
	\pair{\mu(1)\otimes \bigotimes_{b\in\Out(i)} \beta^{(b)}}{\tau }_{GL_a} B^{(i)}_\tau 
	\prod_{b\in \Out(i)} s_{\beta^{(b)}}[X^{(hb)}]^\perp.
\end{align*}
Since $\mu(1)$ and the $\beta^{(b)}$ are partitions, the tensor product is a polynomial character and we may assume that $\tau$ is a partition with at most $a$ rows. We have
\begin{align*}
	H^{(i,a)}_{\mu(1)} &= \sum_{\substack{\tau\in \bY_a \\ \beta^\bullet\in \bY^{\Out(i)}}} t^\beta 
	\pair{\mu(1)\otimes \bigotimes_{b\in\Out(i)} \beta^{(b)}}{\tau}_{GL_a} B^{(i)}_\tau \prod_{b\in\Out(i)} s_{\beta^{(b)}}[X^{(hb)}]^\perp.
\end{align*}
Therefore
\begin{align*}
	H^{\bi,\ba}_{\mu(\bullet)} &= H^{(i,a)}_{\mu(1)}(H^{\hbi,\hba}_{\hmu(\bullet)}) \\
	&= \sum_{\gd} \cK^{\hbi,\hba}_{\gd,\hmu(\bullet)}(t_{Q_1}) \sum_{\tau , \beta^\bullet}   t^\beta 
	\pair{\mu(1)\otimes \bigotimes_{b\in\Out(i)} \beta^{(b)}}{\tau }_{GL_a}  \\ &\qquad
	B^{(i)}_\tau\prod_{b\in\Out(i)} s_{\beta^{(b)}}[X^{(hb)}]^\perp(s_{\gd}) \\
	\cK^{\bi,\ba}_{\lad,\mu(\bullet)}(t_{Q_1}) &=
	\sum_{\gd} \cK^{\hbi,\hba}_{\gd,\hmu(\bullet)}(t_{Q_1}) \sum_{\tau, \beta^\bullet}   t^\beta   \pair{\mu(1)\otimes  \bigotimes_{b\in\Out(i)} \beta^{(b)}}{\tau}_{GL_a}\\ &\qquad \pair{B^{(i)}_\tau 
		\prod_{b\in \Out(i)} s_{\beta^{(b)}}[X^{(hb)}]^\perp(s_{\gd})}{s_\lad} \\
	&=	\sum_{\gd} \cK^{\hbi,\hba}_{\gd,\hmu(\bullet)}(t_{Q_1}) \sum_{\tau, \beta^\bullet}   t^\beta   \pair{\mu(1)\otimes  \bigotimes_{b\in\Out(i)} \beta^{(b)}}{\tau}_{GL_a}\\ &\qquad \prod_{j\in Q_0} \pair{B_\tau^{\delta_{ij}} \prod_{i \overset{b}{\rightarrow}j} 
		s_{\beta^{(b)}}^\perp(s_{\gamma^{(j)}})}{s_{\la^{(j)}}}.
\end{align*}

For the rest of this computation, for simplicity let us assume that $Q$ has no loop at $i$. Then the $i$-th factor in the product over $Q_0$ is $\pair{B_\tau(s_{\gamma^{(i)}})}{s_{\la^{(i)}}}$. Let us assume this is nonzero.
By the definition of $\tau$, the pair $(\tau, \gamma^{(i)})$ must be of the form $(\alpha(w),\zeta(w))$ where $\alpha(w)$ and $\zeta(w)$ are the first $a$ and last $n-a$ parts of the weight $w^{-1}(\la^{(i)}+\rho_n)-\rho_n$ where $n=\ell(\la^{(i)})$ and $w$ is in the set $S_n^a$ of minimum length coset representatives for $S_n/(S_a\times S_{n-a})$. Moreover the $GL_a$-pairing vanishes unless $\mu(1)\subset \tau =\alpha(w)$. We have
\begin{align*}
	\cK^{\bi,\ba}_{\lad,\mu(\bullet)}(t_{Q_1}) &=
	\sum_{\substack{w\in S_n^a \\ \alpha(w) \supset \mu(1)}}
	(-1)^w \sum_{\gd} \cK^{\bi,\ba}_{\gd,\hmu(\bullet)}(t_{Q_1}) \sum_{\beta^\bullet\in \bY^{\Out(i)}} t^\beta \\
	&\qquad \pair{\mu(1)\otimes  \bigotimes_{b\in\Out(i)} \beta^{(b)}}{\alpha(w)}_{GL_a}
	\prod_{j\in Q_0\setminus\{i\}} \pair{s_{\gamma^{(j)}}}{s_{\la^{(j)}} \prod_{i \overset{b}{\rightarrow}j} s_{\beta^{(b)}}}.
\end{align*}

\begin{ex} \label{X:the example} Let $Q=\bZ/2\bZ$ be the cyclic quiver on two vertices and let:
\begin{align*}
\bi &=(0,1,0,1,0)\\
\ba &=(2,2,2,2,3)\\
\mu(\bullet) &=((4,2),(0,0),(2,2),(0,0),(2,1,1))\\
\lad &= ((6,3,3,1,1),\emptyset).
\end{align*}
The nonzero terms correspond to the $w\in \{ \id , s_2 \}$. For $w=\id$ we have
	$(\alpha,\gamma)=((6,3),(3,1,1))$. Letting $U=(u_1,u_2)$ we compute $s_{\alpha/\mu(1)}=s_{(6,3)/(4,2)}[U] = s_{3}[U]+s_{21}[U]$. We sum over $\gamma^{(1)}$ containing $\la^{(1)}=\emptyset$ with at most $a_1=2$ rows, of the same size as $\alpha/\mu(1)$ (which has size $3$).
	So $\gamma^{(1)} \in \{(3,0),(2,1)\}$.
	
	For $w=s_2$ we have $(\alpha,\gamma)=((6,2),(4,1,1))$,
	$s_{\alpha/\mu(1)}[U]=s_2[U]$, so $\gamma^{(1)}\in \{(2,0)\}$. Computing recursively, we have
	\begin{align*}
		\cK^{\bi,\ba}_{(6,3,3,1,1),\emptyset)}(t_{01},t_{10}) &=
		t_{01}^3 \cK^{\hbi,\hba}_{((3,1,1),(3,0)),\hmu(\bullet)}(t_{01},t_{10}) \\
		&+ t_{01}^3 \cK^{\hbi,\hba}_{((3,1,1),(2,1)),\hmu(\bullet)}(t_{01},t_{10}) \\
		&- t_{01}^2 \cK^{\hbi,\hba}_{((4,1,1),(2,0)),\hmu(\bullet)}(t_{01},t_{10}) \\
		&= t_{01}^3 t_{10}^3 (t^3+2t^2) + t_{01}^3 t_{10}^3 (t^3+4t^2+2t) - t_{01}^2 t_{10}^2 (t^3+t^2) \\
		&= (t^6+2t^5) + (t^6+4t^5+2t^4) - (t^5+t^4) \\
		&= 2t^6 + 5t^5 + t^4.
	\end{align*}
\end{ex}

\subsection{Nonbranching quivers}

Say that a quiver is nonbranching if every vertex has at most one incoming arrow and at most one outgoing arrow. This is equivalent to requiring that every connected component be a directed path or directed cycle. There is no loss of generality in assuming the quiver is connected. 

The following two Propositions hold by Lemma \ref{L:cycle grading}.

\begin{prop} \label{P:acyclic} Suppose $Q$ is acyclic.
	Then $\cK^{\bi,\ba}_{\lad,\mu(\bullet)}(t_{Q_1})$ is a nonnegative integer times a single monomial.
\end{prop}

\begin{prop} \label{P:cyclic form}
	Let $Q$ be the cyclic quiver $\bZ/r\bZ$. Then there is a unique Laurent monomial of the form $\prod_{i=0}^{r-2} t_{i,i+1}^{a_i}$ and a unique polynomial $\cK^\red_{\lad,\mu(\bullet)}(t)\in \Z[t]$ with $t=t_{01}t_{12}\dotsm t_{r-1,0}$ such that
	\begin{align}\label{E:single loop}
		\cK_{\lad,\mu(\bullet)}^{\bi,\ba}(t_{Q_1}) = \left(\prod_{i=0}^{r-2} t_{i,i+1}^{a_i}\right)
		\cK^\red_{\lad,\mu(\bullet)}(t).
	\end{align}
	We call $\cK^\red_{\lad,\mu(\bullet)}(t)$ the reduced Kostka-Shoji polynomial.
\end{prop}

A $Q_0$-multitableau $T^\bullet$ is a tuple of semistandard tableaux $T^{(i)}$ for $i\in Q_0$.

Say that the $Q_0$-multitableau $T^\bullet$ admits $\cat_{(i,a,\mu)}$ if
\begin{enumerate}
	\item The restriction of $T^\bullet$ to the entries $\{1,2,\dotsc,a\}$ is the Yamanouchi tableau $Y_\mu$. In this case
	let $T_n$ be the first $a_1$ rows of $T^{(i)}-Y_\mu$ and $T_s$ the rest.
	\item If $\Out(i)=\emptyset$ then $T_n=\emptyset$.
\end{enumerate}
If $T^\bullet$ admits $\cat_{(i,a,\mu)}$ then 
$\cat_{(i,a,\mu)}(T^\bullet)$ is the $Q_0$-multitableau obtained as follows.
\begin{itemize}
	\item The $i$-th tableau is replaced by $T_s$. 
	\item If $\{(i,j)\}=\Out(i)$ (including the case $i=j$) then $T_n$ is Schensted column inserted into the $j$-th tableau. 
	\item The other tableaux remain the same.
\end{itemize}
Say that a multitableau $T^\bullet$ is $(\bi,\ba,\mu(\bullet))$-catabolizable if either $(\bi,\ba,\mu(\bullet))$ is the empty triple and $T^\bullet$ is the empty multitableau or $T^\bullet$ admits $\cat_{(i_1,a_1,\mu(1))}$ and 
$\cat_{(i_1,a_1,\mu(1))}(T^\bullet)$ is $(\hbi,\hba,\hmu(\bullet))$-catabolizable with respect to the alphabet $\{a_1+1,a_1+2,\dotsc\}$ where $\hbi,\hba,\hmu(\bullet)$ are obtained from $\bi,\ba,\mu(\bullet)$ by removing $i_1,a_1,\mu(1)$.

The weight $\wt(T^\bullet)$ of a $(\bi,\ba,\mu(\bullet))$-catabolizable multitableau $T^\bullet$ is by definition the product of the monomials $t_{i,j}^{|T_n|}$ that count the number of moving entries in the applications of the $\cat_{(i_k,a_k,\mu(k))}$.

\begin{conj} \label{conj:cat}
	For $Q_0$ having at most one incoming and one outgoing arrow from each vertex, and for a triple $(\bi,\ba,\mu(\bullet))$ such that $\mu(\bullet)$ is $(\bi,\ba)$-dominant,
	\begin{align}\label{E:cj}
		\cK^{\bi,\ba}_{\lad,\mu(\bullet)}(t_{Q_1}) = 
		\sum_{T^\bullet} \wt(T)
	\end{align}
	where $T^\bullet$ runs over the $(\bi,\ba,\mu(\bullet))$-catabolizable multitableaux of shape $\lad$.
\end{conj}

\begin{rem} 
	\begin{enumerate}
		\item If $Q_0$ is a single loop then Conjecture \ref{conj:cat} reduces to \cite[Conjecture 27]{SW}. If in addition every partition $\mu(k)$ is a rectangle, there are several explicit combinatorial formulas for $\cK^{\bi,\ba}_{\lad,\mu(\bullet)}(t_{Q_1})$, the closest in spirit to catabolizable being \cite[Theorem 21]{S:multi}. The precise catabolizable condition was proved for rectangles of a fixed width \cite[Proposition 24]{S:multi}, where the polynomials are the graded isotypic components of the $GL_n\times \bC^*$-module given by the coordinate ring of a nilpotent adjoint orbit closure in $\mathfrak{gl}_n$. When the common width of rectangles is $1$ we obtain the cocharge Kostka polynomials $\tilde{K}_{\la^t,\mu}(t)=t^{n(\mu)}K_{\la^t,\mu}(t^{-1})$ where $n(\mu)=\sum_{i=1}^{\ell(\mu)} (i-1)\mu_i$ \cite{Las} \cite{S:multi}.
		For the easier class of Littlewood-Richardson tableaux (again for rectangles) one obtains the one-dimensional sums from type $A_{n-1}^{(1)}$ Kirillov-Reshetikhin crystals \cite{KiSh} \cite{S:poset} \cite{S:affine} \cite{ScW}. Finally, when every $\mu(k)$ is a single row, we obtain the Kostka-Foulkes polynomial $K_{\la,\mu}(t)$ \cite{LS}.
		\item For a cyclic quiver with $r=2$ vertices, the special case of Conjecture~\ref{conj:cat} in the setting of Example~\ref{X:cycle} and with $\mu(\bullet)$ concentrated at a single vertex was stated and proved in \cite{ShL}. See Example~\ref{X:cyclic vanishing} for more remarks on the cyclic quiver setting (for $r\ge 2$).
		\item For the $A_2$ quiver ($Q_0=\{0,1\}$ and $Q_1 = \{(0,1)\}$), with $\bi=(0,1,0,1,\dotsc)$, $\ba=(1,1,1,1,\dotsc)$, and $\mu(k)$ a single row partition for all $k$, an explicit combinatorial formula for $\cK^{\bi,\ba}_{\lad,\mu(\bullet)}(t_{01})$ was proved in \cite{Cr}.
	\end{enumerate}
\end{rem}

\subsection{An example}
We believe that the construction for the cyclic quiver sheds light on the single loop case of parabolic Kostka polynomials. The main new feature is that the intermediate results of catabolism produce meaningful terms in the quiver Kostka-Shoji polynomials in which the first current $(i_1,a_1,\mu(1))$ has been removed. 

\begin{ex}\label{X:tableaux}
	Take the data of Example \ref{X:the example}. It is convenient to have the following canonical tableaux in mind. They are of the shapes of $\mu(\bullet)$, where $a=10$ and $b=11$ as in hexadecimal notation:
	\begin{equation}
	\tableau[sty]{1&1&1&1\\2&2} \quad \varnothing \quad \tableau[sty]{5&5\\6&6}\quad\varnothing \quad\tableau[sty]{9&9\\a\\b}
	\end{equation}
	
	For $\gd = (3,1,1)\otimes(3)$ we have bitableaux
	\begin{equation*}
		\tableau[sty]{5&5&9\\6\\9} \otimes \tableau[sty]{6&a&b}\,, \quad
		\tableau[sty]{5&5&a\\6\\b} \otimes \tableau[sty]{6&9&9}\,, \quad
		\tableau[sty]{5&5&b\\6\\9} \otimes \tableau[sty]{6&9&a}
	\end{equation*}
	which have weights $t_{10}^3 t^2$, $t_{10}^3 t^3$, and $t_{10}^3 t^2$ respectively. Under the isomorphism
	\begin{equation}\label{E:skew iso}
	s_{\tableau[pby]{ & \\ \\ }}[U] + s_{\tableau[pby]{&&}}[U]\cong
	s_{\tableau[pby]{\fl&\fl&\fl&\fl& & \\ \fl&\fl& }}[U]  
	\end{equation}
	each tableau at vertex $1$ factors as
	\begin{equation*}
		\tableau[scy]{x&y&z} \mapsto 	\tableau[scy]{\fl&\fl&\fl&\fl& y& z\\ \fl&\fl& x} 
	\end{equation*}
	and therefore these produce (by reversing $\cat_{0,2,(4,2)}$) bitableaux with empty factor at vertex $1$ and the following factors at vertex $0$:
	\begin{equation*}
		\tableau[sby]{1&1&1&1&a&b\\ 2&2&6 \\ 5&5&9 \\ 6 \\ 9} \quad
		\tableau[sby]{1&1&1&1&9&9\\ 2&2&6 \\ 5&5&a \\ 6 \\ b} \quad
		\tableau[sby]{1&1&1&1&9&a\\ 2&2&6 \\ 5&5&b \\ 6 \\ 9}
	\end{equation*}
	These three bitableaux have weights $t^5,t^6,t^5$ respectively. 
	
	For $\gd = ((3,1,1),(2,1))$ we have 7 bitableaux:
	\begin{equation*}
		\tableau[sty]{5&5&a\\6\\b} \otimes \tableau[sty]{6&9\\9}\,, \quad
		\tableau[sty]{5&5&b\\6\\b} \otimes \tableau[sty]{6&9\\a}\,, \quad
		\tableau[sty]{5&5&b\\6\\b} \otimes \tableau[sty]{6&a\\9}
	\end{equation*}
	\begin{equation*}
		\tableau[sty]{5&9&a\\6\\b} \otimes \tableau[sty]{5&9\\6}\,, \qquad
		\tableau[sty]{5&9&b\\6\\9} \otimes \tableau[sty]{5&a\\6}
	\end{equation*}
	\begin{equation*}
		\tableau[sty]{5&5&9\\6\\9} \otimes \tableau[sty]{6&a\\b}\,, \qquad
		\tableau[sty]{5&5&9\\6\\b} \otimes \tableau[sty]{6&9\\a}
	\end{equation*}
	Under \eqref{E:skew iso}
	each tableau at vertex $1$ factors as
	\begin{equation*}
		\tableau[scy]{x&y\\ z} \mapsto 
		\tableau[scy]{\fl&\fl&\fl&\fl& x& y\\ \fl&\fl& z}.
	\end{equation*}
	This produces the following ``tableaux" at vertex $0$; the empty tableau is put at vertex $1$ in all cases.
	\begin{equation*}
		\tableau[sty]{1&1&1&1&6&9\\2&2&9\\5&5&a\\6\\b}\quad
		\tableau[sty]{1&1&1&1&6&9\\2&2&a\\5&5&b\\6\\b}\quad
		\tableau[sty]{1&1&1&1&6&a\\2&2&9\\5&5&b\\6\\b}
	\end{equation*}
	\begin{equation*}
		\tableau[sty]{1&1&1&1&5&9\\2&2&6\\5&9&a\\6\\b}\quad
		\tableau[sty]{1&1&1&1&5&a\\2&2&6\\5&9&b\\6\\9}
	\end{equation*}
	\begin{equation*}
		\tableau[sty]{1&1&1&1&6&a\\2&2&\tf{b}\\5&5&\tf{9}\\6\\9} \qquad
		\tableau[sty]{1&1&1&1&6&9\\2&2&\tf{a}\\5&5&\tf{9}\\6\\b}
	\end{equation*}
	The last two are not valid tableaux as they fail semistandardness.
	The weights of the other five bitableaux are respectively $t^5,t^5,t^4,t^6,t^5$.
	
	Finally, to account for the two bitableaux that must be canceled, consider $\gd=((4,1,1),(2))$. For $\cK^{\hbi,\hba}_{\gd,\hmu(\bullet)}(t_{01},t_{10})$ we have the bitableaux
	\begin{equation*}
		\tableau[sty]{5&5&9&b\\6\\9} \otimes \tableau[sty]{6&a}\qquad
		\tableau[sty]{5&5&9&a\\6\\b} \otimes \tableau[sty]{6&9}
	\end{equation*}
	producing bitableaux which are empty at vertex $1$ and at vertex $0$ are given by
	\begin{equation*}
		\tableau[sby]{1&1&1&1&6&a\\2&2\\ 5&5&\tf{9}&\tf{b}\\6\\9}\qquad
		\tableau[sby]{1&1&1&1&6&9\\2&2\\ 5&5&\tf{9}&\tf{a}\\6\\b}
	\end{equation*}
	The row-reading words of these two objects are Knuth-equivalent to the row-reading words of the two nontableaux above.
\end{ex}

\section{Shuffle products and $K$-theoretic Hall algebras}\label{S:shuffle}

Let us now explain how our constructions relate to shuffle products and $K$-theoretic Hall algebras.

\subsection{Shuffle products}
Recall the definition of $\hR(\cG_{\nu^\bullet})$ from Section~\ref{SS:torus weights} and let
\begin{align}
\cS_Q &= \bigoplus_{\nu^\bullet\in\Z_{\ge 0}^{Q_0}} R(\cG_{\nu^\bullet})\\
\hcS_Q &= \bigoplus_{\nu^\bullet\in\Z_{\ge 0}^{Q_0}} \hR(\cG_{\nu^\bullet}) = \cS_Q\otimes_{\Z[t_{Q_1}^{\pm 1}]}\Z((t_{Q_1})).
\end{align}
We consider two shuffle products
\begin{align}
* : \cS_Q \times \cS_Q \rightarrow \cS_Q\\
\star : \hcS_Q \times \hcS_Q \rightarrow \hcS_Q
\end{align}
which are $R(T^{Q_1})$-linear (resp., $\mathrm{Frac}(R(T^{Q_1}))$-linear for $\star$) and are defined on elements
\begin{align*}
f\in R(G_{\alpha^\bullet}) &=\Z[(u^{(i)}_k)^{\pm 1} \mid i\in Q_0, 1\le k\le \alpha^{(i)}]^{S_{\alpha^\bullet}}\\
g\in R(G_{\beta^\bullet}) &= \Z[(v^{(i)}_k)^{\pm 1} \mid i\in Q_0, 1\le \ell\le\beta^{(i)}]^{S_{\beta^\bullet}}
\end{align*}
as follows:
\begin{align}\label{E:shuffle-ast}
f * g &=D_{w_0^\bullet}\left(f(u)g(v)\prod_{\substack{b\in Q_1\\k,\ell}}(1-t_{b}v_k^{(tb)}/u_\ell^{(hb)})\right)\in R(\cG_{\alpha^\bullet+\beta^\bullet})\\
f \star g &= D_{w_0^\bullet}\left(f(u)g(v)\prod_{\substack{b\in Q_1\\k,\ell}}\frac{1}{1-t_{b}u^{(tb)}_k/v^{(hb)}_\ell}\right)\in \hR(\cG_{\alpha^\bullet+\beta^\bullet})
\end{align}
where $k,\ell$ in the products run over $1\le k\le \beta^{(tb)}$ and $1\le \ell\le \alpha^{(hb)}$ in $*$ and $1\le k\le \alpha^{(tb)}$ and $1\le \ell\le \beta^{(hb)}$ in $\star$. We identify $R(G_{\alpha^\bullet+\beta^\bullet})$ with Laurent polynomials which are symmetric in the variables $$\{u^{(i)}_k,v^{(i)}_\ell \mid 1\le k\le \alpha^{(i)},1\le \ell\le \beta^{(j)}\}$$ for each $i\in Q_0$. Accordingly, $D_{w_0^\bullet}=\prod_{i\in Q_0}D_{w_0^{(i)}}$ is the Demazure symmetrizer~\eqref{E:Dw0} with respect to each of these sets of variables.

One may verify directly from the definitions that $\ast$ and $\star$ are associative. They are related as follows. Consider the maps
\begin{align}
R(\cG_{\nu^\bullet}) &\overset{\tau_{\nu^{\bullet}}}{\longrightarrow} \hR(\cG_{\nu^\bullet}),\qquad
 [W] \mapsto [W\otimes_{\bC} \Sym(E_{\nu^\bullet}^\vee)]
\end{align}
Then one has a commutative diagram:
\begin{center}
\begin{tikzcd}
R(\cG_{\alpha^\bullet})\times \arrow{d}[swap]{\tau_{\alpha^\bullet}\times\tau_{\beta^\bullet}} R(\cG_{\beta^\bullet}) \arrow[r,"*"] & R(\cG_{\alpha^\bullet+\beta^\bullet})\arrow{d}{\tau_{\alpha^\bullet+\beta^\bullet}}\\
\hR(\cG_{\alpha^\bullet})\times \hR(\cG_{\beta^\bullet}) \arrow[r,"\star"] & \hR(\cG_{\alpha^\bullet+\beta^\bullet}). 
\end{tikzcd}
\end{center}

\begin{rem}
For cyclic quivers and with all parameters $t_a\equiv t$ equal, the shuffle product $*$ coincides with a specialization of that of Negut \cite{Ne}, which is a ``trigonometric degeneration'' of the original shuffle algebras of Feigin and Odesskii~\cite{FO}. For a more transparent match with \eqref{E:shuffle-ast}, see also \cite[Proof of Proposition 2.3]{Ne-rev}, where the relevant specialization is $q_1=q=0$ and $q_2=t$. Similar connections exist in the Jordan quiver case of \cite{SV} \cite{FT}. However, an important distinction must be made: these works involve the study of a subalgebra $\cS_Q^+\subset \cS_Q$ consisting of symmetric Laurent polynomials satisfying a ``wheel condition.'' In \cite{Ne}, it is  ultimately shown that the Drinfeld double of $\cS_Q^+$ for the cyclic quiver with $r$ vertices is isomorphic to the quantum toroidal algebra $\Utor$.
\end{rem}

\begin{rem}
For general quivers the shuffle product $*$ is a $K$-theoretic variant of the product in the cohomological Hall algebra \cite[Theorem 2.2]{KS}, with additional equivariance coming from the torus $T^{Q_1}$. We will make this connection more precise in Section \ref{S:K-HA} below. In Section~\ref{SS:qt-preprojective} we will also consider the $(q,t)$-version of the shuffle product $\star$, which corresponds to the preprojective $K$-theoretic Hall algebra~\cite{YZ} and that of \cite{Ne} for cyclic quivers and \cite{SV} \cite{FT} for the Jordan quiver.
\end{rem}

\subsubsection{Quiver Hall-Littlewood series as a shuffle product}

Fix data $(\bi,\ba,\mu(\bullet))$ and define variables $u=u(1)+u(2)+\dotsm+u(m)$ exactly as in Section~\ref{SS:big partial flag}, using $u$ instead of $x$. These are the auxiliary variables of Section~\ref{SS:quiver-sym-fun}. Let $\chi^{\bi,\ba}_{\mu(\bullet)}(u)$ be the quiver Hall-Littlewood series in these variables. For each $1\le k\le m$ let $\al^\bullet(k)=\nu^\bullet((i_k),(a_k))$ be the dimension vector with $\alpha^{(i)}(k)=a_k\delta_{i,i_k}$.

Using formula \eqref{E:chi-D} for the quiver Hall-Littlewood series, it is easy to prove the following by induction on $m$:

\begin{prop}\label{P:quiver-kostka-shuffle}
For any $(\bi,\ba,\mu(\bullet))$, the quiver Hall-Littlewood series is given by a shuffle product:
\begin{align}
\chi^{\bi,\ba}_{\mu(\bullet)}(u) &= s_{\mu(1)}[u(1)] \star \dotsm \star s_{\mu(m)}[u(m)]
\end{align}
where each $s_{\mu(k)}[u(k)]$ is understood as an element of $R(G_{\alpha^\bullet(k)})$.
\end{prop}

More generally, for any triples $(\bi',\ba',\mu'(\bullet))$ and $(\bi'',\ba'',\mu''(\bullet))$, let $\bi=\bi'\bi''$, $\ba=\ba'\ba''$, $\mu(\bullet)=\mu'(\bullet)\mu''(\bullet)$ be given by concatenation. Then one can show that
\begin{align}\label{E:chi-concat}
\chi^{\bi,\ba}_{\mu(\bullet)} = \chi^{\bi',\ba'}_{\mu'(\bullet)} \star \chi^{\bi'',\ba''}_{\mu''(\bullet)}
\end{align}
where $\chi^{\bi',\ba'}_{\mu'(\bullet)}\in \hR(\cG_{\nu^\bullet(\bi',\ba')})$ and $\chi^{\bi'',\ba''}_{\mu''(\bullet)}\in \hR(\cG_{\nu^\bullet(\bi'',\ba'')})$.

\begin{ex}\label{ex:chi-shuffle}
Consider the single-arrow quiver: $Q_0=\{0,1\}$ and $Q_1=\{(0,1)\}$. Take $\bi=(0,0,1)$ and $\ba=(2,1,1)$, so that:
\begin{align*}
u(1) = \{u^{(0)}_1,u^{(0)}_2\},\qquad
u(2) = \{u^{(0)}_3\},\qquad 
u(3) = \{u^{(1)}_1\}.
\end{align*}
Then for any $\mu(\bullet)$ we have
\begin{align*}
\chi_{\mu(\bullet)}^{\bi,\ba}(u) &= D_{w_0^\bullet}\left(\frac{u(1)^{\mu(1)}u(2)^{\mu(2)}u(3)^{\mu(3)}}{(1-t_{01}u^{(0)}_1/u_1^{(1)})(1-t_{01}u^{(0)}_2/u_1^{(1)})(1-t_{01}u^{(0)}_3/u^{(1)}_1)}\right)\\
&= s_{\mu(1)}[u(1)] \star D_{w_0^\bullet}^{2,3}\left(u(2)^{\mu(2)}u(3)^{\mu(3)}(1-t_{01}u^{(0)}_3/u^{(1)}_1)^{-1}\right)\\
&= s_{\mu(1)}[u(1)] \star s_{\mu(2)}[u(2)] \star s_{\mu(3)}[u(3)]
\end{align*}
where $D_{w_0^\bullet}$ is the Demazure symmetrizer with respect to all variables, while $D_{w_0^\bullet}^{2,3}$ is that only for the variables $u(2)+u(3)$. Here we use well-known properties such as $D_{w_0^\bullet}=D_{w_0^\bullet}D_{w_0^\bullet}^{2,3}$ as operators and that $D_{w_0^\bullet}^{2,3}$ commutes with the $u(1)$ variables. We note that the second equality corresponds to \eqref{E:chi-concat} with $\bi'=(0),\ba'=(2)$ and $\bi''=(0,1),\ba''=(1,1)$.
\end{ex}

Proposition~\ref{P:quiver-kostka-shuffle} leads to another interpretation of the quiver Kostka-Shoji polynomials $\cK_{\lad,\mu(\bullet)}^{\bi,\ba}(t_{Q_1})$: they are the structure constants of iterated products in the shuffle algebra $\hcS_Q$ with respect to Schur polynomials, each supported at a single vertex.

\subsubsection{Connection to quiver currents}\label{SS:current-shuffle}

More generally, we can strengthen Proposition~\ref{P:quiver-kostka-shuffle} to connect the action of a quiver current on $\Lambda^Q$ to a corresponding shuffle product. Let $i\in Q_0$, $a\in\Z_{>0}$, and $\mu\in X_+(GL_a)$. We will consider the action of $H^{(i,a)}_\mu$ on a tensor Schur function $s_{\xi^\bullet}[X^\bullet]$.

Let $\alpha^\bullet = \nu^\bullet((i),(a))$ be the dimension vector with $\alpha^{(i)}=a$ and $\alpha^{(j)}=0$ otherwise, and let $\beta^\bullet$ be an arbitrary dimension vector. Let $U=(u^{(i)}_1,\dotsc,u^{(i)}_a)$ and $V=(v^{(i)}_k)_{i\in Q_0,1\le k\le\beta^{(i)}}$ be auxiliary variables. Let $v^{(i)}=(v^{(i)}_k)_{1\le k\le\beta^{(i)}}$ be the variables of $V$ at vertex $i$.

The Proposition below shows that the action $H^{(i,a)}_\mu$ is a lifting to quiver symmetric functions of the operation of shuffle product by $s_\mu[U]$, regarded as an element of $R(G_{\alpha^\bullet})$. More generally, for any $\beta^\bullet$, there is a natural restriction map $\Lambda^Q\to R(\cG_{\beta^\bullet})$, denoted $f\mapsto f[V]$ for variables $V$ as above. A product of Schur functions $s_{\xi^\bullet}[X^\bullet]$ is sent to the product $s_{\xi^\bullet}[V]=\prod_{i\in Q_0}s_{\xi^{(i)}}(v^{(i)}_1,\dotsc,v^{(i)}_{\beta^{(i)}})$ of corresponding Schur polynomials in the variables of $V$. 

Recall from Theorem~\ref{T:quiver HL functions and Kostka} that $\bY^{Q_0}_{\nu^\bullet}$ denotes the set of $Q_0$-tuples of partitions $\lad$ such that $\la^{(i)}$ has at most $\nu^{(i)}$ rows for $i\in Q_0$.

\begin{prop}\label{P:quiver-current-shuffle}
For any $\xi^\bullet\in \bY^{Q_0}_{\beta^\bullet}$ and $\lambda^\bullet\in\bY^{Q_0}_{\alpha^\bullet+\beta^\bullet}$, the coefficient of $s_\lad[X^\bullet]$ in the Schur expansion of
\begin{align}
H^{(i,a)}_\mu \cdot s_{\xi^\bullet}[X^\bullet]
\end{align}
is equal to the coefficient of $s_\lad[U+V]$ in the Schur expansion of the shuffle product $s_\mu[U]\star s_{\xi^\bullet}[V]$. (The former coefficient is zero unless $\lambda^\bullet\in\bY^{Q_0}_{\alpha^\bullet+\beta^\bullet}$.)
\end{prop}

\begin{proof}
We follow the proof of Theorem~\ref{T:quiver HL functions and Kostka}, except that we consider the action of $H^{(i,a)}(U)$ on the generating function
\begin{align*}
\prod_{j\in Q_0}\Omega[v^{(j)}X^{(j)}] = \sum_{\xi^\bullet\in\bY^{Q_0}_{\beta^\bullet}} s_{\xi^\bullet}[V]s_{\xi^\bullet}[X^\bullet]. 
\end{align*}
This is given by
\begin{align*}
&H^{(i,a)}(U) \cdot \prod_{j\in Q_0}\Omega[v^{(j)}X^{(j)}]\\
&= R(U)\Omega[-U^*v^{(i)}] \Omega[UX^{(i)}]\prod_{j\in Q_0}\Omega[v^{(j)}X^{(j)}]\prod_{b\in \Out(i)} \Omega[t_{b} U^* v^{(hb)}].
\end{align*}
Let $Z$ be the ordered union of the variables $U$ and $V$, with $U$ before $V$. Write $\mu+\xi^\bullet$ for the concatenated weight with respect this ordering, so that $z^{\mu+\xi^\bullet}=u^\mu v^{\xi^\bullet}$. Define $\rho^\bullet$ and $S^\bullet$ with respect to $Z$. Then
\begin{align*}
&\pair{H^{(i,a)}_{\mu}\cdot s_{\xi^\bullet}[X^\bullet]}{s_\lad[X^\bullet]}\\
&=[u^{\mu}s_{\xi^\bullet}[V]]\pair{H^{(i,a)}(U)\cdot\prod_{j\in Q_0}\Omega[v^{(j)}X^{(j)}]}{s_\lad[X^\bullet]} \\
&= [u^{\mu}s_{\xi^\bullet}[V]] R(U)\Omega[-U^*v^{(i)}]
	s_\lad[U+V] \prod_{b\in\Out(i)} \Omega[t_{b} U^* v^{(hb)}]\\
&= [u^{\mu}v^{\xi^\bullet}] R(U)R(V)\Omega[-U^*v^{(i)}]
	s_\lad[U+V] \prod_{b\in\Out(i)} \Omega[t_{b} U^* v^{(hb)}]\\
&= [z^{\mu+\xi^\bullet}] R(Z)s_\lad[Z] \prod_{b\in\Out(i)} \Omega[t_{b} U^* v^{(hb)}]\\
&= \sum_{\wbul\in S^\bullet} (-1)^{\wbul} [ z^{\mu+\xi^\bullet+\rho^\bullet-\wbul(\lad+\rhod)}] \prod_{b\in\Out(i)} \Omega[t_{b} U^* v^{(hb)}]\\
&= \sum_{\wbul\in S^\bullet} (-1)^{\wbul} [ z^{(w^{\bullet})^{-1}(\mu+\xi^\bullet+\rho^\bullet)-(\lad+\rhod)}] \prod_{b\in\Out(i)} (\wbul)^{-1}\left(\Omega[t_{b} U^* v^{(hb)}]\right)\\
&= \sum_{\wbul\in S^\bullet} (-1)^{\wbul} [ z^{\lad+\rhod-w^{\bullet}(\mu+\xi^\bullet+\rho^\bullet)}] \prod_{b\in\Out(i)} \wbul\left(\Omega[t_{b} U (v^{(hb)})^*]\right)\\
&= [z^{\lad+\rhod}]J^\bullet\left(z^{\mu+\xi^\bullet+\rho^\bullet}\prod_{b\in\Out(i)} \Omega[t_{b} U (v^{(hb)})^*]\right)
\end{align*}
which is equal to the coefficient of $s_\lad[U+V]$ in $s_\mu[U]\star s_{\xi^\bullet}[V]$.
\end{proof}

\begin{rem}\label{R:shuffle-act}
Consider the $\Z[t_{Q_1}^{\pm 1}]$-subalgebra $\hcS'_Q\subset \hcS_Q$ generated under $\star$ by all $R(G_{\alpha^\bullet})$ such that $\alpha^\bullet$ is supported at a single vertex. Using Proposition~\ref{P:quiver-current-shuffle} and the associativity of $\star$, one can show that the assignment $s_\mu[U]\mapsto H^{(i,a)}_\mu$ (with the former interpreted as in Proposition~\ref{P:quiver-current-shuffle}) extends to an action of $\hcS'_Q$ on $\Lambda^Q$.
\end{rem}

\subsection{$K$-theoretic Hall algebra}\label{S:K-HA}

Kontsevich and Soibelman \cite{KS} defined cohomological Hall algebras, which give a ring structure to the space $\oplus_{\nu^\bullet\in\Z_{\ge 0}^{Q_0}}H^*_{G_{\nu^\bullet}}(E_{\nu^\bullet})$. We adapt their definition to the setting of equivariant $K$-theory to define a product on $\oplus_{\nu^\bullet\in\Z_{\ge 0}^{Q_0}}K^{\cG_{\nu^\bullet}}(E_{\nu^\bullet})$. We call this the $K$-theoretic Hall algebra. 

\subsubsection{Definition} Fix dimension vectors $\alpha^\bullet,\beta^\bullet\in\Z_{\ge 0}^{Q_0}$. The product 
\begin{align*}
\circledast : K^{\cG_{\al^\bullet}}(E_{\alpha^\bullet}) \times K^{\cG_{\beta^\bullet}}(E_{\beta^\bullet}) \to K^{\cG_{\al^\bullet+\beta^\bullet}}(E_{\alpha^\bullet+\beta^\bullet}) 
\end{align*}
is defined as the composition
\begin{align*}
\cF_1 \circledast \cF_2 &= (q_*\boxtimes\mathrm{id})\circ \mathrm{ind}\circ i_*\circ \pi^*\circ \mathrm{for}(\cF_1\boxtimes\cF_2)
\end{align*}
of the following maps:
\begin{enumerate}
\item the outer tensor product identification:
\begin{align*}
K^{\cG_{\al^\bullet}}(E_{\alpha^\bullet}) \times K^{\cG_{\beta^\bullet}}(E_{\beta^\bullet}) \cong K^{\cG_{\al^\bullet}\times \cG_{\beta^\bullet}}(E_{\alpha^\bullet}\times E_{\beta^\bullet}).
\end{align*}

\item forgetting to the diagonal subgroup $T^{Q_1}=\Delta T^{Q_1}\subset T^{Q_1}\times T^{Q_1}$:
\begin{align*}
K^{\cG_{\al^\bullet}\times \cG_{\beta^\bullet}}(E_{\alpha^\bullet}\times E_{\beta^\bullet})\overset{\mathrm{for}}{\longrightarrow} K^{G_{\al^\bullet}\times G_{\beta^\bullet}\times T^{Q_1}}(E_{\alpha^\bullet}\times E_{\beta^\bullet}).
\end{align*}

\item pullback and pushforward along the natural maps
\begin{align*}
E_{\alpha^\bullet}\times E_{\beta^\bullet} \overset{\pi}{\longleftarrow} E_{\alpha^\bullet,\beta^\bullet} \overset{i}{\longrightarrow} E_{\alpha^\bullet+\beta^\bullet} 
\end{align*}
where $E_{\alpha^\bullet,\beta^\bullet}$ is the space of quiver representations on the space $\bV_{\alpha^\bullet}\oplus\bV_{\beta^\bullet}$ leaving $\bV_{\beta^\bullet}$ stable; here $E_{\alpha^\bullet+\beta^\bullet}$ consists of quiver representations on the space $\bV_{\alpha^\bullet}\oplus\bV_{\beta^\bullet}$, $i$ is the natural inclusion map from {\em lower} block triangular matrices, and $\pi$ is the projection to the diagonal blocks of such matrices.

\item the induction map $$K^{\cG_{\alpha^\bullet,\beta^\bullet}}(E_{\alpha^\bullet+\beta^\bullet})\overset{\mathrm{ind}}{\longrightarrow} K^{\cG_{\alpha^\bullet+\beta^\bullet}}(\cG_{\alpha^\bullet+\beta^\bullet}\times^{\cG_{\alpha^\bullet,\beta^\bullet}}E_{\alpha^\bullet+\beta^\bullet})$$
where $\cG_{\alpha^\bullet,\beta^\bullet}=G_{\al^\bullet}\times G_{\beta^\bullet}\times T^{Q_1}$.

\item the pushforward $q_*\boxtimes\mathrm{id}$ arising from the isomorphism
\begin{align*}
\cG_{\alpha^\bullet+\beta^\bullet}\times^{\cG_{\alpha^\bullet,\beta^\bullet}}E_{\alpha^\bullet+\beta^\bullet} \cong (\cG_{\alpha^\bullet+\beta^\bullet}/\cG_{\alpha^\bullet,\beta^\bullet})\times E_{\alpha^\bullet+\beta^\bullet}
\end{align*}
where $q: \cG_{\alpha^\bullet+\beta^\bullet}/\cG_{\alpha^\bullet,\beta^\bullet}\to \mathrm{pt}$.
(The isomorphism results from the fact that $E_{\alpha^\bullet+\beta^\bullet}$ carries a $\cG_{\alpha^\bullet+\beta^\bullet}$-action extending that of $\cG_{\alpha^\bullet,\beta^\bullet}$.)
\end{enumerate}

One may check using this definition that $\circledast$ is associative as in \cite[\S2.3]{KS}. Alternatively, we will compute $\circledast$ more explicitly in the next subsection and see that it is given by the shuffle product $*$.

\subsubsection{Explicit formula}
The explicit computation of $\circledast$ proceeds as follows. By the Thom isomorphism (e.g., \cite[Theorem 5.4.7]{CG}), we have
\begin{align}
K^{\cG_{\nu^\bullet}}(E_{\nu^\bullet}) \cong K^{\cG_{\nu^\bullet}}(\mathrm{pt}) \cong R(\cG_{\nu^\bullet})
\end{align}
for any dimension vector $\nu^\bullet$, via the pullback $\pi^*:K^{\cG_{\nu^\bullet}}(\mathrm{pt})\to \K^{\cG_{\nu^\bullet}}(E_{\nu^\bullet})$ where $\pi: E_{\nu^\bullet}\to \mathrm{pt}$. The inverse is given by $\pi_*([\cO_{E_{\nu^\bullet}}])^{-1}\pi_*$ due to the projection formula (\cite[(5.3.13)]{CG} or \cite[Exercise III.8.3]{Ha}):
\begin{align}
R^i\pi_*(\pi^*(\cF)\otimes \cO_{E_{\nu^\bullet}})\cong \cF\otimes R^i\pi_*\cO_{E_{\nu^\bullet}}
\end{align}
for any $\cG_{\nu^\bullet}$-equivariant locally free sheaf $\cF$ on $E_{\nu^\bullet}$. We note that $R^i\pi_*=0$ for $i>0$ and we have a natural identification $\pi_*\cO_{E_{\nu^\bullet}} \cong \Sym(E^\vee_{\nu^\bullet})$. Putting all this together, we obtain an isomorphism
\begin{align}
\psi: K^{\cG_{\nu^\bullet}}(E_{\nu^\bullet})&\to K^{\cG_{\nu^\bullet}}(\mathrm{pt}) \cong R(\cG_{\nu^\bullet})\\
[\cF]&\mapsto \frac{\chi_{\cG_{\nu^\bullet}}(E_{\nu^\bullet},[\cF])}{\gch_{\cG_{\nu^\bullet}}(\Sym(E^\vee_{\nu^\bullet}))}.\notag
\end{align}
More generally, for a vector space $V$ with an action of a linear algebraic group $G$, we have an isomorphism $\psi:K^G(V)\to R(G)$ given by $\chi_G(V,-)/\gch_G(\Sym(V^\vee))$. We abuse notation and use the symbol $\psi$ repeatedly below, with $G$ and $V$ determined by context.

\begin{prop}
For any dimension vectors $\alpha^\bullet,\beta^\bullet$, the following diagram commutes:
\begin{center}
\begin{tikzcd}
K^{\cG_{\alpha^\bullet}}(E_{\alpha^\bullet}) \arrow{d}[swap]{\psi\times\psi} \times K^{\cG_{\beta^\bullet}}(E_{\beta^\bullet}) \arrow[r,"\circledast"] & K^{\cG_{\alpha^\bullet+\beta^\bullet}}(E_{\alpha^\bullet+\beta^\bullet}) \arrow[d,"\psi"]\\
R(\cG_{\alpha^\bullet})\times R(\cG_{\beta^\bullet}) \arrow[r,"*"] & R(\cG_{\alpha^\bullet+\beta^\bullet})
\end{tikzcd}
\end{center}
where $*$ is the shuffle product \eqref{E:shuffle-ast}. 
\end{prop}

\begin{proof}
We have the following commutative diagram:
\begin{center}
\begin{tikzcd}
K^{\cG_{\alpha^\bullet,\beta^\bullet}}(E_{\alpha^\bullet}\times E_{\beta^\bullet}) \arrow[d,"\psi"]\arrow[r,"\pi^*"] & K^{\cG_{\alpha^\bullet,\beta^\bullet}}(E_{\alpha^\bullet,\beta^\bullet}) \arrow[d,"\psi"]\arrow[r,"i_*"] & K^{\cG_{\alpha^\bullet,\beta^\bullet}}(E_{\alpha^\bullet+\beta^\bullet}) \arrow[d,"\psi"]\\
R(\cG_{\alpha^\bullet,\beta^\bullet}) \arrow[r,"\mathrm{id}"] & R(\cG_{\alpha^\bullet,\beta^\bullet}) \arrow[r,"\overline{i}_*"] & R(\cG_{\alpha^\bullet,\beta^\bullet})
\end{tikzcd}
\end{center}
where we compute using the Koszul complex \cite[\S5.4]{CG} that:
\begin{align*}
\overline{i}_*(f(u,v)) &= f(u,v)\prod_{\substack{b\in Q_1\\k,\ell}}(1-t_{b}v^{(tb)}_k/u^{(hb)}_\ell).
\end{align*}
As above, $1\le k\le \beta^{(tb)}$ and $1\le \ell\le \alpha^{(hb)}$ in the product. Note that this product is nothing but $1/\gch_{\cG_{\alpha^\bullet,\beta^\bullet}}\Sym(E_{\alpha^\bullet+\beta^\bullet}/E_{\alpha^\bullet,\beta^\bullet})^\vee$.

We have another commutative diagram involving the induction from $\cG_{\alpha^\bullet,\beta^\bullet}$ to $\cG_{\alpha^\bullet+\beta^\bullet}$:
\begin{center}
\begin{tikzcd}
K^{\cG_{\alpha^\bullet,\beta^\bullet}}(E_{\alpha^\bullet+\beta^\bullet}) \arrow{d}[swap]{\psi}\arrow[r,"\mathrm{ind}"] & K^{\cG_{\alpha^\bullet+\beta^\bullet}}(\cG_{\alpha^\bullet+\beta^\bullet}\times^{\cG_{\alpha^\bullet,\beta^\bullet}}E_{\alpha^\bullet+\beta^\bullet}) \arrow[r,"q_*\boxtimes \mathrm{id}"] & K^{\cG_{\alpha^\bullet+\beta^\bullet}}(E_{\alpha^\bullet+\beta^\bullet}) \arrow[d,"\psi"]\\
R(\cG_{\alpha^\bullet,\beta^\bullet}) \arrow[rr,"D_{w_0^\bullet}"] & & R(\cG_{\alpha^\bullet+\beta^\bullet})
\end{tikzcd}
\end{center}
Note that $\cG_{\alpha^\bullet+\beta^\bullet}/\cG_{\alpha^\bullet,\beta^\bullet}$ is product of partial flag varieties. This allows one to compute the map along the bottom arrow using the Borel-Weil-Bott theorem. The relevant computation is that
\begin{align*}
\chi(\cG_{\alpha^\bullet+\beta^\bullet}\times^{\cG_{\alpha^\bullet,\beta^\bullet}}E_{\alpha^\bullet+\beta^\bullet}, \cG_{\alpha^\bullet+\beta^\bullet}\times^{\cG_{\alpha^\bullet,\beta^\bullet}} V) &= (q\times\pi)_*[\cG_{\alpha^\bullet+\beta^\bullet}\times^{\cG_{\alpha^\bullet,\beta^\bullet}} V] \\
&= (q\times 1)_*(1\times\pi)_*[\cG_{\alpha^\bullet+\beta^\bullet}\times^{\cG_{\alpha^\bullet,\beta^\bullet}} V]\\
&= D_{w_0^\bullet}(\chi_{\cG_{\alpha^\bullet,\beta^\bullet}} V)
\end{align*}
for any $\cG_{\alpha^\bullet,\beta^\bullet}$-equivariant vector bundle $V$ on $E_{\alpha^\bullet+\beta^\bullet}$.
\end{proof}


\begin{rem}\label{R:lusztig-conv}
Lusztig \cite{L:conv} gives a closely related construction of a product in his geometric realization of the negative part of quantum enveloping algebras. In fact, even though Lusztig's construction uses perverse sheaves, one can apply it equally well to equivariant $K$-theory (of coherent sheaves). This direction is pursued in unpublished work of Grojnowski \cite{G}. It is not difficult to check that Lusztig's product in the $K$-theoretic setting agrees with the product $\circledast$ defined above.
\end{rem}

\subsubsection{Pushforward classes}
Fix data $(\bi,\ba,\mu(\bullet))$ and recall the second projection $\Spr_{\bi,\ba} : Z_{\bi,\ba}\to E_{\nu^\bullet(\bi,\ba)}$.
Consider the class $(\Spr_{\bi,\ba})_*[\cW^{\mu(\bullet)}]$. Its image in $R(\cG_{\nu^\bullet(\bi,\ba)})$ is given by
\begin{align}
\psi^{\bi,\ba}_{\mu(\bullet)} &= \psi((\Spr_{\bi,\ba})_*[\cW^{\mu(\bullet)}])\notag\\
&= (\gch_{\cG_{\nu^\bullet(\bi,\ba)}} \Sym(E^\vee_{\nu^\bullet(\bi,\ba)}))^{-1}\chi^{\bi,\ba}_{\mu(\bullet)}\\
&= D_{w_0^\bullet}\left(x^{\mu(\bullet)}(\gch_{\cG_{\nu^\bullet(\bi,\ba)}} \Sym(E^\vee_{\nu^\bullet(\bi,\ba)}))^{-1}B_{\bi,\ba}\right).\notag
\end{align}
Note that all of the denominators of $B_{\bi,\ba}$ are canceled in this expression, which results in $\psi^{\bi,\ba}_{\mu(\bullet)}\in R(\cG)$.

\begin{rem}
The elements $\psi^{\bi,\ba}_{\mu(\bullet)}\in R(\cG)$ are a generalization of Hall-Littlewood $R$-polynomials. In particular, for $Q$ equal to the Jordan quiver, $\ba=(1,1,\dotsc,1)$, and $\mu(\bullet)=(\mu(1),\mu(2),\dotsc,\mu(m))=:\mu$, one has that $\psi^{\bi,\ba}_{\mu(\bullet)}$ is equal to
\begin{align*}
R_\mu(u_1,\dotsc,u_m;t) = \sum_{w \in S_m} w\left(u_1^{\mu_1}\dotsm u_m^{\mu_m}\prod_{i<j}\frac{u_i-tu_j}{u_i-u_j}\right).
\end{align*}
\end{rem}

\begin{rem}
For $Q$ equal to the cyclic quiver and with data $\bi,\ba$ as in Example~\ref{X:cycle}, the $\psi^{\bi,\ba}_{\mu(\bullet)}$ are closely related but not identical to Shoji's polynomials $R_{\mu^{\bullet}}^+(z;\mathbf{t})$ \cite[(4.1.2)]{Sh3}. Taking the same data $\bi,\ba$ for the opposite cyclic quiver results in polynomials closely related to Shoji's $R_{\mu^{\bullet}}^-(z;\mathbf{t})$.
\end{rem}

The $\psi^{\bi,\ba}_{\mu(\bullet)}$ behave well with respect to the shuffle product, which one can expect by the nature of Lusztig's construction \cite{L:conv} (see Remark~\ref{R:lusztig-conv}). Given triples $(\bi',\ba',\mu'(\bullet))$ and $(\bi'',\ba'',\mu''(\bullet))$, let $\bi=\bi'\bi''$, $\ba=\ba'\ba''$, $\mu(\bullet)=\mu'(\bullet)\mu''(\bullet)$ be their concatenations. Then one can easily verify (or deduce from \eqref{E:chi-concat}) that:
\begin{align}
\psi^{\bi,\ba}_{\mu(\bullet)} &= \psi_{\mu'(\bullet)}^{\bi',\ba'} * \psi_{\mu''(\bullet)}^{\bi'',\ba''}.
\end{align}
where $\psi_{\mu'(\bullet)}^{\bi',\ba'}\in R(\cG_{\nu^\bullet(\bi',\ba')})$ and $\psi_{\mu''(\bullet)}^{\bi'',\ba''}\in R(\cG_{\nu^\bullet(\bi'',\ba'')})$

\subsection{$(q,t)$-currents and shuffles}\label{SS:qt-preprojective}

Given a quiver $Q$, let $\overline{Q}$ be the doubled quiver obtained by adding the opposite of each arrow in $Q$. Let $Q_1$ denote the original arrows in $Q$ and $Q_1'$ their opposites in $\overline{Q}$. We assume that all arrow variables $t_{b}\equiv t$ from $Q_1$ are equal and we let $q$ be an arrow variable for the opposite arrows. We will define a double version $\overline{H}^{(i,a)}(U)$ of the quiver currents, which act on quiver symmetric functions $\Lambda^{Q}$ over the ring $\Z[q^{\pm 1},t^{\pm 1}]$. When $q=0$, the double currents reduce to our original quiver currents $H^{(i,a)}(U)$ at $t_{b}\equiv t$.

We use the notation of Section~\ref{SS:quiver-sym-fun}. For any $i\in Q_0$ and $a\in\Z_{>0}$,  the $(q,t)$-quiver current $\overline{H}^{(i,a)}(U)=\sum_{\beta\in\bZ^a} u^\beta \overline{H}^{(i,a)}_{\beta}$ is defined by the following formula:
\begin{align*}
R(U) \Omega[UX^{(i)}] \Omega[-U^*X^{(i)}(1+qt)]^\perp
    \prod_{a\in\Out'(i)}\Omega[q U^*X^{(ha)}]^\perp
    \prod_{b\in\Out(i)}\Omega[t U^*X^{(hb)}]^\perp
\end{align*}
where $\Out(i)=\{b\in Q_1 \mid tb=i\}$ as before and $\Out'(i)=\{a\in Q_1' \mid ta=i \}$.

\begin{rem}
When $Q$ is the Jordan quiver and $a=1$, these are the symmetric function operators $D_k$ of \cite{GHT}, up to a plethystic minus sign. In the case of a cyclic quiver with $r\ge 2$ vertices, the $(q,t)$-currents with $a=1$ arise naturally in the vertex representation \cite{Sa} of quantum toroidal $\fsl_r$ (see Appendix~\ref{S:toroidal}).
\end{rem}

Recall the conventions of \S\ref{SS:current-shuffle}. Following \cite[Corollary 3.6]{YZ}, we define for any dimension vectors $\alpha^\bullet$, $\beta^\bullet$, the shuffle product:
\begin{align}
\label{E:qt-shuffle-star}
f \ \overline{\star}\ g &= D_{w_0^\bullet}\left(f(u)g(v)\frac{\prod_{\substack{i\in Q_0\\k,\ell}}(1-qt u_k^{(i)}/v_\ell^{(i)})}{\prod_{\substack{a \in Q_1'\\k,\ell}}(1-q u^{(ta)}_k/v^{(ha)}_\ell)\prod_{\substack{b\in Q_1\\k,\ell}}(1-tu^{(tb)}_k/v^{(hb)}_\ell)}\right).
\end{align}
This belongs to $R(G_{\alpha^\bullet+\beta^\bullet})\otimes_{\Z}\Z((q,t))$. The geometrically-defined algebra associated with this shuffle product (or rather, its $\ast$-variant as in \eqref{E:shuffle-ast}) is the {\em preprojective} $K$-theoretic Hall algebra, as developed in \cite[\S3.2 and \S4.1]{YZ}.

The proof of Proposition~\ref{P:quiver-current-shuffle} carries through in this setting to show that the $(q,t)$-currents $\overline{H}^{(i,a)}(U)$ provide the same kind of symmetric function lifting for the shuffle product $\overline{\star}$:
\begin{prop}
\label{P:qt-quiver-current-shuffle}
Assume the setup of Proposition~\ref{P:quiver-current-shuffle}. For any $\xi^\bullet\in \bY^{Q_0}_{\beta^\bullet}$ and $\lambda^\bullet\in\bY^{Q_0}_{\alpha^\bullet+\beta^\bullet}$, the coefficient of $s_\lad[X^\bullet]$ in the Schur expansion of
\begin{align}
\overline{H}^{(i,a)}_\mu \cdot s_{\xi^\bullet}[X^\bullet]
\end{align}
is equal to the coefficient of $s_\lad[U+V]$ in the Schur expansion of the shuffle product $s_\mu[U]\ \overline{\star}\ s_{\xi^\bullet}[V]$.
\end{prop}

\begin{rem}
Similarly as in Remark~\ref{R:shuffle-act}, one can show using Proposition~\ref{P:qt-quiver-current-shuffle} that the assignment $s_\mu[U]\mapsto \overline{H}^{(i,a)}_\mu(U)$ extends to an action on $\Lambda^Q$ by the algebra generated under $\overline{\star}$ by the $R(G_{\alpha_\bullet})$ with $\alpha^\bullet$ supported at a single vertex.
\end{rem}

\appendix
\section{Vertex representation of quantum toroidal $\fsl_r$}\label{S:toroidal}

In this appendix we assume that $Q$ is the cyclic quiver with vertices $Q_0=\Z/r\Z$ and arrows $Q_1=\{(i,i+1) \mid i\in Q_0\}$, where $r\ge 2$. Our aim is to relate the $(q,t)$-quiver currents $\overline{H}^{(i,1)}(u)$ in one variable ($a=1$) to the vertex representation~\cite{Sa} of the quantum toroidal algebra $\Utor$. These currents are given explicity as
\begin{align}
\overline{H}^{(i,1)}(u) &= \Omega\left[uX^{(i)}\right]\Omega\left[-u^{-1}\left(X^{(i)}-tX^{(i+1)}-qX^{(i-1)}+qtX^{(i)}\right)\right]^\perp.\label{E:qt-current}
\end{align}

\subsection{Notation}
Let $\{\al_i \mid i\in\Z/r\Z\}$ and $\{\al_i^\vee \mid i\in\Z/r\Z\}$ be the standard simple roots and simple coroots for $\widehat{\fsl_r}$. Let $\pair{\cdot}{\cdot}$ be the canonical pairing between coroots and roots. Set $\alb_i = \al_i$ and $\alb_i^\vee = \al_i^\vee$ for $i\neq 0$ and
\begin{align*}
\alb_0 &=-\sum_{i\neq 0}\alb_i,\qquad
\alb_0^\vee =-\sum_{i\neq 0}\alb_i^\vee.
\end{align*}
Let $P$ and $L$ be the weight lattice and root lattice of $\widehat{\fsl_r}$, and let $\Pb,\Lb$ be the corresponding lattices for $\fsl_r$.

\subsection{Quantum toroidal $\fsl_{r}$}

Quantum toroidal algebras were introduced in \cite{GKV} for any semisimple Lie algebra, where they were defined using a presentation inspired by the Drinfeld presentation of quantum affine algebras~\cite{Be}. We will consider only the quantum toroidal algebra $\Utor$ of type $\fsl_{r}$, which is the associative algebra over $\mathbb{K}=\C(q^{\frac{1}{2}},t)$ generated by
\begin{align*}
&E^{(i)}_{k}\, ,\, F^{(i)}_{k}&&\text{for $i\in \Z/r\Z$ and $k\in\Z$}\\
&K^{(i)}_{\pm,k} &&\text{for $i\in \Z/r\Z$, $k\in\Z_{\ge 0}$, and $\pm=+,-$}
\end{align*}
subject to certain relations. To state the relations, we collect the generators into currents:
\begin{align}
E^{(i)}(u) &= \sum_{k\in\Z}E^{(i)}_{k}u^{-k}\\
F^{(i)}(u) &= \sum_{k\in\Z}F^{(i)}_{k}u^{-k}\\
K^{(i)}_\pm(u) &= \sum_{k\ge 0}K^{(i)}_{\pm,k}u^{\mp k}.
\end{align}
(For the purposes of our discussion, we ignore the elements $q^{d_1},q^{d_2}$ and we specialize the central element $q^{\frac{1}{2}c}\mapsto q^{\frac{1}{2}}$.)

For any $m\in\Z$, let
\begin{align}
\theta_m(u) &= \frac{q^m u-1}{u-q^m}
\end{align}
and let $M=(m_{ij})_{i,j\in\Z/r\Z}$ be the matrix given by:
\begin{align}\label{E:M}
m_{ij} = \begin{cases}\pm 1&\text{if $i=j\pm 1$}\\0 &\text{otherwise.}\end{cases}
\end{align}
Then the defining relations of $\Utor$ are as follows for all $i,j\in\Z/r\Z$:
\begin{align}
&K^{(i)}_{+,0} K^{(i)}_{-,0} = K^{(i)}_{-,0} K^{(i)}_{+,0} = 1\\
&K^{(i)}_\pm(u)K^{(j)}_\pm(v) = K^{(j)}_\pm(v)K^{(i)}_\pm(u)\label{E:K-same}\\
&\theta_{-\pair{\al_i^\vee}{\al_j}}(q^{-1}t^{m_{ij}}u/v)K^{(i)}_-(u)K^{(j)}_+(v) = \theta_{-\pair{\al_i^\vee}{\al_j}}(q t^{m_{ij}}u/v)K^{(j)}_+(v)K^{(i)}_-(u)\label{E:K-opp}\\
&K^{(i)}_\pm(u)E^{(j)}(v) = \theta_{\mp\pair{\al_i^\vee}{\al_j}}(q^{-\frac{1}{2}}(t^{-m_{ij}}v/u)^{\pm 1})E^{(j)}(v)K^{(i)}_\pm(u)\\
%
%
&K^{(i)}_\pm(u)F^{(j)}(v) = \theta_{\pm\pair{\al_i^\vee}{\al_j}}(q^{\frac{1}{2}}(t^{-m_{ij}}v/u)^{\pm 1})F^{(j)}(v)K^{(i)}_\pm(u)\\
%
%
&(t^{m_{ij}}u-q^{\pair{\al_i^\vee}{\al_j}}v)E^{(i)}(u)E^{(j)}(v) = (q^{\pair{\al_i^\vee}{\al_j}}t^{m_{ij}}u-v)E^{(j)}(v)E^{(i)}(u)\\
&(t^{m_{ij}}u-q^{-\pair{\al_i^\vee}{\al_j}}v)F^{(i)}(u)F^{(j)}(v) = (q^{-\pair{\al_i^\vee}{\al_j}}t^{m_{ij}}u-v)F^{(j)}(v)F^{(i)}(u)\\
&[E^{(i)}(u),F^{(j)}(v)] = \frac{\delta_{ij}}{q-q^{-1}}(\delta(qv/u)K^{(i)}_+(q^{\frac{1}{2}}v)-\delta(qu/v)K^{(i)}_-(q^{\frac{1}{2}}u))
\end{align}
where $\delta(x)=\sum_{n\in\Z} x^n$, and the following quantum Serre relations:
\begin{align}
\sum_{\sigma\in S_2} \Big(&E^{(i)}(u_{\sigma(1)})E^{(i)}(u_{\sigma(2)})E^{(i\pm 1)}(v)
-(q+q^{-1})E^{(i)}(u_{\sigma(1)})E^{(i\pm 1)}(v)E^{(i)}(u_{\sigma(2)})\\
&\quad +E^{(i\pm 1)}(v)E^{(i)}(u_{\sigma(1)})E^{(i)}(u_{\sigma(2)})\Big)=0\notag\\
\sum_{\sigma\in S_2}\Big(&F^{(i)}(u_{\sigma(1)})F^{(i)}(u_{\sigma(2)})F^{(i\pm 1)}(v)-(q+q^{-1})F^{(i)}(u_{\sigma(1)})F^{(i\pm 1)}(v)F^{(i)}(u_{\sigma(2)})\\
&\quad +F^{(i\pm 1)}(v)F^{(i)}(u_{\sigma(1)})F^{(i)}(u_{\sigma(2)})\Big)=0\notag.
\end{align}

\subsection{Heisenberg algebra}

We can replace the generators $K^{(i)}_{\pm,k}$ of $\Utor$ by other generators $P^{(i)}_{\pm k}$ for $i\in\Z/r\Z$ and $k>0$ by imposing the following equality of formal series:
\begin{align}
K^{(i)}_{\pm}(u) &= K^{(i)}_{\pm,0}\exp\left(\pm(q-q^{-1})\sum_{k\ge 1} P^{(i)}_{\pm k} u^{\mp k}\right).
\end{align}
Then relations \eqref{E:K-same} and \eqref{E:K-opp} are then equivalent to the following Heisenberg-type relations for the $P^{(i)}_{\pm k}$:
\begin{align}\label{E:heis1}
[P^{(i)}_{k},P^{(i)}_{-k}] &= \frac{1}{k}\frac{(q^k-q^{-k})(q^{2k}-q^{-2k})}{(q-q^{-1})^2}\\
[P^{(i)}_{k},P^{(i+1)}_{-k}] &= \frac{t^{-k}}{k}\frac{(q^k-q^{-k})^2}{(q-q^{-1})^2}\label{E:heis2}\\
[P^{(i)}_{k},P^{(i-1)}_{-k}] &= -\frac{t^k}{k}\frac{(q^k-q^{-k})^2}{(q-q^{-1})^2}\label{E:heis3}
\end{align}
and all other $[P^{(i)}_{k},P^{(j)}_{l}]=0$.

The Fock space representation of the Lie algebra spanned by the $P^{(i)}_{\pm k}$ is multisymmetric functions $\Lambda^{Q}=\otimes_{i\in \Z/r\Z}\Lambda^{(i)}$, where:
\begin{align}
\label{E:heis-act+}
P^{(i)}_{k} &\mapsto \frac{p_k[(X^{(i)}(q+q^{-1})-X^{(i+1)}t^{-1}-X^{(i)}t)(q-q^{-1})]^\perp}{k(q-q^{-1})}\\
\label{E:heis-act-}
P^{(i)}_{-k} &\mapsto \frac{p_k[X^{(i)}(q-q^{-1})]}{k(q-q^{-1})}.
\end{align}
This is verified using the following relation between operators on $\Lambda$:
\begin{align}
[p_k^\perp,p_k] = k.
\end{align}
In other words, $p_k^\perp = k\frac{\del}{\del p_k}$.

\subsection{Skew group algebra}
The definition of the vertex representation of $\Utor$ requires an auxiliary algebra $\K\{\Lb\}$, which is a skew-version of the group algebra $\K[\Lb]$. 

\newcommand{\es}{\mathsf{e}}

The weight lattice $\Pb$ of $\fsl_r$ has a $\Z$-basis given by the fundamental weights $\{\Lambdab_i \mid i\neq 0\}$. Another basis of $\Pb$ is $\{\alb_2,\dotsc,\alb_{r-1},\Lambdab_{r-1}\}$. Define $\K\{\Pb\}$ to be the $\K$-algebra generated by $\es^{\pm \alb_2},\dotsc,\es^{\pm\alb_{r-1}},\es^{\pm\Lambdab_{r-1}}$ subject to the relations:
\begin{align}
\es^{\alb_i}\es^{-\alb_i} &= \es^{-\alb_i}\es^{\alb_i}=1\\
\es^{\Lambdab_{r-1}}\es^{-\Lambdab_{r-1}} &= \es^{-\Lambdab_{r-1}}\es^{\Lambdab_{r-1}} = 1\\
\es^{\alb_i}\es^{\alb_j} &= (-1)^{\pair{\al_i^\vee}{\al_j}}\es^{\alb_j}\es^{\alb_j}\\
\es^{\alb_i}\es^{\Lambdab_{r-1}} &= (-1)^{\pair{\alb_i^\vee}{\Lambdab_{r-1}}}\es^{\Lambdab_{r-1}}\es^{\alb_i}
\end{align}
where $i,j\neq 0,1$. For any element $\overline{\beta}=m_2\alb_2+\dotsm+m_{r-1}\alb_{r-1}+m\Lambdab_{r-1}\in\Pb$, we define a monomial
\begin{align}\label{E:skew-monomial}
\es^{\overline{\beta}}=(\es^{\alb_2})^{m_2}\dotsm (\es^{\alb_{r-1}})^{m_{r-1}} (\es^{\Lambdab_{r-1}})^m.
\end{align}
These monomials form a $\K$-basis for $\K\{\Pb\}$.

\comment{
We note that
\begin{align*}
\alb_1 &= -2\alb_2-\dotsm-(r-2)\alb_{r-2}-(r-1)\alb_{r-1}+r\Lambdab_{r-1}
\end{align*}
and hence
\begin{align*}
\es^{\alb_1}\es^{\alb_j} &= (-1)^{\pair{\alb_1^\vee}{\alb_j}}\es^{\alb_j}\es^{\alb_1},\qquad 
(\es^{\alb_1})^{-1} = \es^{-\alb_1}.
\end{align*}
}

Let $\K\{\Lb\}$ be the subalgebra of $\K\{\Pb\}$ generated by $\es^{\pm\alb_i}$ for $i\neq 0$.
The defining relations of $\K\{\Lb\}$ with respect to these generators are for all $i,j\neq 0$:
\begin{align}
\es^{\alb_i}\es^{-\alb_i} &= \es^{-\alb_i}\es^{\alb_i} = 1\\
\es^{\alb_i}\es^{\alb_j} &= (-1)^{\pair{\alb_i^\vee}{\alb_j}}\es^{\alb_j}\es^{\alb_i}.
\end{align}
However, we stress that all monomials $\es^{\overline{\beta}}$ in $\K\{\Pb\}$ and $\K\{\Lb\}$ are defined via \eqref{E:skew-monomial}.

\comment{
\begin{ex}
Suppose that $r=4$. Then $\alb_1 = -2\alb_2-3\alb_3+4\Lambdab_3$ and
\begin{align*}
\es^{\alb_1}\es^{\alb_2} &= \es^{-2\alb_2}\es^{-3\alb_3}\es^{4\Lambdab_3}\es^{\alb_2}\\
&=-\es^{-\alb_2}\es^{-3\alb_3}\es^{4\Lambdab_3}\\
&= -\es^{\alb_1+\alb_2}.
\end{align*}
Because of such sign differences, we choose to follow \cite{Sa} and use \eqref{E:skew-monomial} to define all monomials.
\end{ex}
}

We regard elements of $\K\{\Lb\}$ as operators on $\K\{\Lb\}$ acting by left multiplication. We introduce additional operators $u^{P^{(i)}_0}$ for $i\in\Z/r\Z$ acting from $\K\{\Lb\}$ to $\K\{\Lb\}[u^{\pm 1}]$ as follows:
\begin{align}
u^{P^{(i)}_0}\cdot \es^{\overline{\beta}} &= u^{\pair{\alb_i^\vee}{\overline{\beta}}}t^{\frac{1}{2}\sum_{j=1}^n \pair{\alb_i^\vee}{m_j\alb_j}M_{ij}}\es^{\overline{\beta}}
\end{align}
where $\overline{\beta}=\sum_{j=1}^{r-1} m_j\alb_j$ and $M$ is the matrix \eqref{E:M}. Equivalently, to compute the action of $u^{P^{(i)}_0}$, one can use its commutation relations with multiplication operators:
\begin{align}
u^{P^{(i)}_0} \es^{\alb_j} &= \es^{\alb_j}u^{P^{(i)}_0}\times \begin{cases}u^{2}&\text{if $i=j$}\\u^{-1}t^{\mp\frac{1}{2}} &\text{if $i=j\pm 1$}\\1 &\text{otherwise}\end{cases}
\end{align}
where $j\neq 0$, together with its action on $1=\es^0$:
\begin{align}
u^{P^{(i)}_0}\cdot 1 = 1.
\end{align}

Finally, we define operators $q^{\del_{\alb_i}}$ for $i\in\Z/r\Z$ acting on $\K\{\Lb\}$ as follows:
\begin{align}
q^{\del_{\alb_i}}\cdot \es^{\overline{\beta}} &= q^{\pair{\alb_i^\vee}{\overline{\beta}}}\es^{\overline{\beta}}.
\end{align}

\comment{
\begin{align}
q^{\al_i} \es^{\al_i} &= q^{2}\es^{\al_j}q^{\al_i} \\
q^{\al_i} \es^{\al_{i+1}} &= q^{-1} \es^{\al_{i+1}}q^{\al_i}\\
q^{\al_i} \es^{\al_{i-1}} &= q^{-1} \es^{\al_{i-1}}q^{\al_i}
\end{align}
}

\subsection{Vertex representation}

Let $[k]_q = \frac{q^k-q^{-k}}{q-q^{-1}}$ for any $k\in\Z$.

\begin{prop}{\rm ($p=0$ case of \cite[Proposition 3.2.2]{Sa})}\label{P:toroidal-vertex}
The following assignment gives rise to an action of $\Utor$ on $\Lambda^{Q}\otimes \K\{\Lb\}$:
\begin{align}
E^{(i)}(u) &\mapsto \exp\left(\sum_{k\ge 1}\frac{P^{(i)}_{-k}}{[k]_q}(q^{-\frac{1}{2}}u)^k\right)\exp\left(-\sum_{k\ge 1}\frac{P^{(i)}_k}{[k]_q}(q^{\frac{1}{2}}u)^{-k}\right) \otimes \es^{\alb_i}u^{1+P^{(i)}_0}\\
F^{(i)}(u) &\mapsto \exp\left(-\sum_{k\ge 1}\frac{P^{(i)}_{-k}}{[k]_q}(q^{\frac{1}{2}}u)^k\right)\exp\left(\sum_{k\ge 1}\frac{P^{(i)}_k}{[k]_q}(q^{-\frac{1}{2}}u)^{-k}\right)\otimes \es^{-\alb_i}u^{1-P^{(i)}_0}\\
K^{(i)}_\pm(u) &\mapsto \exp\left(\pm(q-q^{-1})\sum_{k\ge 1} P^{(i)}_{\pm k} u^{\mp k}\right) \otimes q^{\pm\del_{\alb_i}}
\end{align}
for any $i\in\Z/r\Z$, where the $P^{(i)}_{\pm k}$ act on $\Lambda^Q$ according to \eqref{E:heis-act+} and \eqref{E:heis-act-}.
\end{prop}

\comment{
\begin{align}
E_i(z) &= \exp\left(\sum_{k\ge 1}H_{i,k} \frac{q-q^{-1}}{q^k-q^{-k}}(q^{-\frac{1}{2}}z)^k\right)\exp\left(-\sum_{k\ge 1}H_{i,-k}\frac{q-q^{-1}}{q^k-q^{-k}}(q^{\frac{1}{2}}z)^{-k}\right)^\perp \es^{\al_i}z^{1+\al_i}\\
\exp\left(\sum_{k\ge 1}\frac{p_k[X^{(i)}]}{k}(q^{-\frac{1}{2}}z)^k\right)\exp\left(-\sum_{k\ge 1}\frac{p_k[(q+q^{-1})X^{(i)}-t^{-1}X^{(i+1)}-tX^{(i)}]}{k}(q^{\frac{1}{2}}z)^{-k}\right)^\perp \es^{\al_i}z^{1+\al_i}
\end{align}
}

\subsection{Connection to quiver currents}
Using \eqref{E:heis-act+} and \eqref{E:heis-act-} we can express the action of $E^{(i)}(q^{\frac{1}{2}}u)$ in Proposition~\ref{P:toroidal-vertex} as follows:
\begin{align*}
&E^{(i)}(q^{\frac{1}{2}}u) \mapsto\\
&\qquad\notag\Omega[uX^{(i)}]\Omega[(1+q^{-2})X^{(i)}-q^{-1}t^{-1}X^{(i+1)}-q^{-1}tX^{(i-1)}]^\perp \otimes \es^{\alb_i}(q^{\frac{1}{2}}u)^{1+P^{(i)}_{0}}.
%
\end{align*}
The connection to the $(q,t)$-quiver currents \eqref{E:qt-current} is then achieved by a simple change of parameters
\begin{align}\label{E:change-params}
\varkappa :\quad q\mapsto q^{-1}t,\quad t\mapsto q^{-1}t^{-1}.
\end{align}
Explicitly, we have
\begin{align}
&\kappa(\overline{H}^{(i,1)}(u)) = (q^{\frac{1}{2}}u)^{-1-P^{(i)}_{0}}\es^{-\alb_i}E^{(i)}(q^{\frac{1}{2}}u).
\end{align}


\begin{thebibliography}{XXXX}

\bibitem[ADK]{ADK}
S. Abeasis, A. Del Fra, and H. Kraft.
The geometry of representations of $A_m$. 
Math. Ann. {256} (1981), no. 3, 401--418.

\bibitem[AH]{AH} P. Achar and A. Henderson. Orbit closures in the enhanced nilpotent cone. Adv. Math. {219} (2008), no. 1, 27--62. 

\bibitem[Be]{Be}
J. Beck.
Braid group action and quantum affine algebras.
Comm. Math. Phys. 165 (1994), no. 3, 555--568.

\bibitem[BGLX]{BGLX}
F. Bergeron, A. Garsia, E. Leven, and G. Xin.
Some remarkable new plethystic operators in the theory of Macdonald polynomials.
J. Comb. 7 (2016), no. 4, 671--714.

\bibitem[BF]{BF}
R. Bezrukavnikov, and M. Finkelberg.
Wreath Macdonald polynomials and the categorical McKay correspondence. 
With an appendix by Vadim Vologodsky. 
Camb. J. Math. 2 (2014), no. 2, 163--190. 

\bibitem[Bro]{Bro}
B. Broer.
Line bundles on the cotangent bundle of the flag variety. 
Invent. Math. 113 (1993), no. 1, 1--20.

\bibitem[Bry]{B}
R. Brylinski. Limits of weight spaces, Lusztig's $q$-analogs, and fiberings of adjoint orbits.
J. Amer. Math. Soc. 2 (1989), no. 3, 517--533.


\bibitem[CG]{CG} N. Chriss and V. Ginzburg. Representation theory and complex geometry. 
Reprint of the 1997 edition. Modern Birkh\"{a}user Classics. Birkh\"{a}user Boston, Inc., Boston, MA, 2010. 

\bibitem[Cr]{Cr} W. Craig. Pictures for A2 Kostka-Shoji numbers.
Undergraduate Thesis---Virginia Tech (2018).

\bibitem[FJMM]{FJMM}
B. Feigin, M. Jimbo, T. Miwa, and E. Mukhin.
Representations of quantum toroidal $\mathfrak{gl}_n$.
J. Algebra 380 (2013), 78--108. 

\bibitem[FO]{FO}
B. L. Feigin and A. V. Odesskii.
Quantized moduli spaces of the bundles on the elliptic curve and their applications. Integrable structures of exactly solvable two-dimensional models of quantum field theory (Kiev, 2000), 123--137, 
NATO Sci. Ser. II Math. Phys. Chem., 35, Kluwer Acad. Publ., Dordrecht, 2001.

\bibitem[FT]{FT}
B. L. Feigin and A. I. Tsymbaliuk.
Equivariant $K$-theory of Hilbert schemes via shuffle algebra.
Kyoto J. Math. 51 (2011), no. 4, 831--854. 

\bibitem[FGT]{FGT}
M. Finkelberg, V. Ginzburg, and R. Travkin. Mirabolic affine Grassmannian and character sheaves. Selecta Math. (N.S.) 14 (2009), no. 3-4, 607--628. 

\bibitem[FI]{FI} M. Finkelberg and A. Ionov. Kostka-Shoji polynomials and Lusztig's convolution diagram. Bulletin of the Institute of Mathematics Academia Sinica (New Series) 13 (2018), no.~1, 31--42.


\bibitem[G]{G} 
A. Garsia. Orthogonality of Milne's polynomials and raising operators. 
Discrete Math. {99} (1992), no. 1-3, 247--264.

\bibitem[GHT]{GHT}
A. M. Garsia, M. Haiman, and G. Tesler.
Explicit plethystic formulas for Macdonald $q,t$-Kostka coefficients.
The Andrews Festschrift (Maratea, 1998). 
Sém. Lothar. Combin. 42 (1999), Art. B42m, 45 pp

\bibitem[GKV]{GKV}
V. Ginzburg, M. Kapranov, and E. Vasserot.
Langlands reciprocity for algebraic surfaces. 
Math. Res. Lett. 2 (1995), no. 2, 147--160.

\bibitem[GN]{GN}
E. Gorsky and A. Negut.
Refined knot invariants and Hilbert schemes.
J. Math. Pures Appl. (9) 104 (2015), no. 3, 403--435. 

\bibitem[Gr]{Gr}
I. Grojnowski.
Affinizing quantum algebras: from $D$-modules to $K$-theory.
Unpublished manuscript (1994).  

\bibitem[Hai]{Hai}
M. Haiman.
Combinatorics, symmetric functions, and Hilbert schemes.
Current developments in mathematics, 2002, 39--111, Int. Press, Somerville, MA, 2003.

\bibitem[Har]{Ha}
R. Hartshorne.
Algebraic geometry. 
Graduate Texts in Mathematics, No. 52. Springer-Verlag, New York-Heidelberg, 1977.

\bibitem[Hu]{H}
Y. Hu. 
Higher cohomology vanishing of line bundles on generalized Springer's resolution.
arXiv:1704.07947 

\bibitem[J]{J} N. Jing. Vertex operators and Hall-Littlewood symmetric functions. 
Adv. Math. {87} (1991), no. 2, 226--248.

\bibitem[KS]{KS} 
M. Kontsevich and Y. Soibelman.
Cohomological Hall algebra, exponential Hodge structures and motivic Donaldson-Thomas invariants. 
Commun. Number Theory Phys. 5 (2011), no. 2, 231--352. 

\bibitem[KSS]{KSS} A. N. Kirillov, A. Schilling, and M. Shimozono.
A bijection between Littlewood-Richardson tableaux and rigged configurations. Selecta Math. (N.S.) 8 (2002), no. 1, 67--135.

\bibitem[KSh]{KiSh} A. Kirillov and M. Shimozono. 
A generalization of the Kostka-Foulkes polynomials. 
J. Algebraic Combin. 15 (2002), no. 1, 27--69. 

\bibitem[LLM]{LLM} L. Lapointe, A. Lascoux, and J. Morse. 
Tableau atoms and a new Macdonald positivity conjecture. Duke Math. J. 116 (2003), no. 1, 103--146.

\bibitem[LM]{LM} L. Lapointe and J. Morse.
Schur function identities, their $t$-analogs, and $k$-Schur irreducibility. Adv. Math. 180 (2003), no. 1, 222--247.

\bibitem[Las]{Las} A. Lascoux.
Cyclic permutations on words, tableaux and harmonic polynomials, 
in: Proc. of the Hyderabad Conference on Algebraic Groups, 1989, Manoj Prakashan, Madras, 1991, 323--347.

\bibitem[LS]{LS}
A. Lascoux and M.-P. Sch\"{u}tzenberger.
Sur une conjecture de H. O. Foulkes.
C. R. Acad. Sci. Paris Sér. A-B {286} (1978), no. 7, A323--A324. 

\bibitem[LSh]{ShL} S. Liu and T. Shoji. Double Kostka polynomials and Hall bimodule. Tokyo J. Math. 39 (2017), no. 3, 743--776. 

\bibitem[LR]{LR} N. Loehr and J. B. Remmel.
A computational and combinatorial exposé of plethystic calculus.
J. Algebraic Combin. {33} (2011), no. 2, 163--198. 

\bibitem[L1]{L:nullcone}
G. Lusztig. 
Green polynomials and singularities of unipotent classes.
Adv. Math. 42 (1981) 169--178.

\bibitem[L2]{L:conv}
G. Lusztig. Quivers, perverse sheaves, and quantized enveloping algebras. J. Amer. Math. Soc. {4} (1991), no. 2, 365--421. 

\bibitem[Mac]{Mac}
I. G. Macdonald. Symmetric functions and Hall polynomials. Second edition. With contributions by A. Zelevinsky. Oxford Mathematical Monographs. Oxford Science Publications. The Clarendon Press, Oxford University Press, New York, 1995.

\bibitem[Nag]{Na}
K. Nagao.
$K$-theory of quiver varieties, $q$-Fock space and nonsymmetric Macdonald polynomials.
Osaka J. Math. 46 (2009), no. 3, 877--907.

\bibitem[Nak]{N}
H. Nakajima. 
Quiver varieties and finite-dimensional representations of quantum affine algebras. 
J. Amer. Math. Soc. 14 (2001), no. 1, 145--238.

\bibitem[Ne1]{Ne-rev}
A. Negut.
The shuffle algebra revisited.
Int. Math. Res. Not. IMRN 2014, no. 22, 6242--6275. 

\bibitem[Ne2]{Ne}
A. Negut.
Quantum Algebras and Cyclic Quiver Varieties. 
Thesis (Ph.D.)--Columbia University. 2015.

\bibitem[OS]{OS:unpublished}
D. Orr and M. Shimozono.
Quiver Hall-Littlewood functions and Kostka-Shoji polynomials. arXiv:1704.05178v1

\bibitem[P]{P}
D. I. Panyushev. 
Generalised Kostka-Foulkes polynomials and cohomology of line bundles on homogeneous vector bundles.
Selecta Math. (N.S.) 16 (2010), no. 2, 315--342.

\bibitem[R]{R}
M. Reineke.
Quivers, desingularizations and canonical bases.
Studies in memory of Issai Schur (Chevaleret/Rehovot, 2000), 325--344, 
Progr. Math., {210}, Birkhäuser Boston, Boston, MA, 2003.

\bibitem[Sa]{Sa}
Y. Saito. Quantum toroidal algebras and their vertex representations.
Publ. Res. Inst. Math. Sci. 34 (1998), no. 2, 155--177. 

\bibitem[Sch]{Sch}
O. Schiffmann.
Quivers of type $A$, flag varieties and representation theory.
Representations of finite dimensional algebras and related topics in Lie theory and geometry, 453--479, 
Fields Inst. Commun., {40}, Amer. Math. Soc., Providence, RI, 2004.

\bibitem[SchV]{SV}
O. Schiffmann and E. Vasserot.
The elliptic Hall algebra and the $K$-theory of the Hilbert scheme of $\mathbb{A}^2$.
Duke Math. J. 162 (2013), no. 2, 279--366.

\bibitem[ScW]{ScW} 
A. Schilling and S. O. Warnaar.
Inhomogeneous lattice paths, generalized Kostka polynomials and $A_{n-1}$ supernomials. Comm. Math. Phys. 202 (1999), no. 2, 359--401. 

\bibitem[S1]{S:poset} M. Shimozono. A cyclage poset structure for Littlewood-Richardson tableaux. European J. Combin. 22 (2001), no. 3, 365--393.

\bibitem[S2]{S:multi} M. Shimozono. Multi-atoms and monotonicity of generalized Kostka polynomials. European J. Combin. 22 (2001), no. 3, 395--414.

\bibitem[S3]{S:affine} M. Shimozono. Affine type A crystal structure on tensor products of rectangles, Demazure characters, and nilpotent varieties. J. Algebraic Combin. 15 (2002), no. 2, 151--187. 

\bibitem[SW]{SW} M. Shimozono and J. Weyman. Graded characters of modules supported in the closure of a nilpotent conjugacy class.  European J. Combin. 21 (2000), no. 2, 257--288. 

\bibitem[SZ1]{SZ} M. Shimozono and M. Zabrocki.
Hall-Littlewood vertex operators and generalized Kostka polynomials.
Adv. Math. {158} (2001), no. 1, 66--85. 

\bibitem[SZ2]{SZ2} M. Shimozono and M. Zabrocki.
Deformed universal characters for classical and affine algebras.
J. Algebra {299} (2006), no. 1, 33--61. 

\bibitem[Sh1]{Sh} T. Shoji. Green Functions Associated to Complex Reflection Groups. J. Algebra {245} (2001), 650--694.

\bibitem[Sh2]{Sh-mac}
T. Shoji. Macdonald functions associated to complex reflection groups. Special issue celebrating the 80th birthday of Robert Steinberg. J. Algebra 260 (2003), no. 1, 426--448. 

\bibitem[Sh3]{Sh2} T. Shoji. Green functions attached to limit symbols. Representation theory of algebraic groups and quantum groups, 443--467, 
Adv. Stud. Pure Math., {40}, Math. Soc. Japan, Tokyo, 2004. 

\bibitem[Sh4]{Sh3} T. Shoji. Kostka functions associated to complex reflection groups and a conjecture of Finkelberg-Ionov. Sci. China Math. 61 (2018), no. 2, 353--384.

\bibitem[T]{T}
A. Tsymbaliuk.
Several realizations of Fock modules for toroidal $\ddot{U}_{q,d}(\mathfrak{sl}_n)$. arXiv:1603.08915

\bibitem[VV]{VV}
M. Varagnolo and E. Vasserot.
On the $K$-theory of the cyclic quiver variety.
Internat. Math. Res. Notices 1999, no. 18, 1005--1028.

\bibitem[W]{W} J. Weyman.
The equations of conjugacy classes of nilpotent matrices.
Invent. Math., 98 (1989), 229--245.

\bibitem[YZ]{YZ}
Y. Yang and G. Zhao.
The cohomological Hall algebra of a preprojective algebra.
Proc. Lond. Math. Soc., to appear. arXiv:1407.7994


\end{thebibliography}
\end{document}